\documentclass[notitlepage,11pt,reqno]{amsart}
\usepackage[foot]{amsaddr}

\usepackage{amssymb,nicefrac,bm,upgreek,mathtools,verbatim}
\usepackage[final]{hyperref}
\usepackage[mathscr]{eucal}
\usepackage{dsfont}
\usepackage[numbers,sort&compress]{natbib}
\usepackage[margin=1in]{geometry}
\allowdisplaybreaks

\usepackage{color}
\definecolor{dmagenta}{rgb}{.4,.1,.5}
\definecolor{dblue}{rgb}{.0,.0,.5}
\definecolor{mblue}{rgb}{.0,.0,.8}
\definecolor{ddblue}{rgb}{.0,.0,.4}
\definecolor{dred}{rgb}{.6,.0,.0}
\definecolor{dgreen}{rgb}{.0,.5,.0}
\definecolor{Eeom}{rgb}{.0,.0,.5}

\usepackage[normalem]{ulem}

\newtheorem{lemma}{Lemma}[section]
\newtheorem{theorem}{Theorem}[section]
\newtheorem{proposition}{Proposition}[section]

\theoremstyle{definition}
\newtheorem{definition}{Definition}[section]

\newtheorem{assumption}{Assumption}[section]

\theoremstyle{remark}
\newtheorem{remark}{Remark}[section]
\numberwithin{equation}{section}

\hypersetup{
  colorlinks=true,
  citecolor=dblue,
  linkcolor=dblue,
  frenchlinks=false,
  pdfborder={0 0 0},
  naturalnames=false,
  hypertexnames=false,
  breaklinks}
\usepackage[capitalize,nameinlink]{cleveref}

\crefname{section}{Section}{Sections}
\crefname{subsection}{Section}{Sections}
\crefname{condition}{Condition}{Conditions}
\crefname{hypothesis}{Hypothesis}{Conditions}
\crefname{assumption}{Assumption}{Assumptions}
\crefname{lemma}{Lemma}{Lemmas}
\crefname{condition}{Condition}{Conditions}

\Crefname{figure}{Figure}{Figures}

\crefformat{equation}{\textup{#2(#1)#3}}
\crefrangeformat{equation}{\textup{#3(#1)#4--#5(#2)#6}}
\crefmultiformat{equation}{\textup{#2(#1)#3}}{ and \textup{#2(#1)#3}}
{, \textup{#2(#1)#3}}{, and \textup{#2(#1)#3}}
\crefrangemultiformat{equation}{\textup{#3(#1)#4--#5(#2)#6}}%
{ and \textup{#3(#1)#4--#5(#2)#6}}{, \textup{#3(#1)#4--#5(#2)#6}}%
{, and \textup{#3(#1)#4--#5(#2)#6}}

\Crefformat{equation}{#2Equation~\textup{(#1)}#3}
\Crefrangeformat{equation}{Equations~\textup{#3(#1)#4--#5(#2)#6}}
\Crefmultiformat{equation}{Equations~\textup{#2(#1)#3}}{ and \textup{#2(#1)#3}}
{, \textup{#2(#1)#3}}{, and \textup{#2(#1)#3}}
\Crefrangemultiformat{equation}{Equations~\textup{#3(#1)#4--#5(#2)#6}}%
{ and \textup{#3(#1)#4--#5(#2)#6}}{, \textup{#3(#1)#4--#5(#2)#6}}%
{, and \textup{#3(#1)#4--#5(#2)#6}}

\crefdefaultlabelformat{#2\textup{#1}#3}
\newcommand{\rc}{{\mathscr{R}}}

\newcommand{\cA}{{\mathscr{A}}}  
\newcommand{\cV}{{\mathcal{V}}}
\newcommand{\fX}{{\mathfrak{X}}}
\newcommand{\fA}{{\mathfrak{A}}}
\newcommand{\eom}{{\mathscr{G}}} 
\newcommand{\cI}{{\mathcal{I}}}  
\newcommand{\cJ}{{\mathcal{J}}}  
\newcommand{\cK}{{\mathcal{K}}}  
\newcommand{\Lg}{\mathcal{L}}    
\newcommand{\cL}{{\mathcal{L}}}  

\newcommand{\cP}{{\mathcal{P}}}  
\newcommand{\cQ}{{\mathcal{Q}}}  

\newcommand{\la}{\lambda}

\newcommand{\ra}{\rightarrow}

\newcommand{\beql}[1]{\begin{equation}\label{#1}}

\newcommand{\beq}{\begin{displaymath}}
\newcommand{\eeqno}{\end{displaymath}}
\newcommand{\eeq}{\end{equation}}

\newcommand{\qasq}{\quad\mbox{as}\quad}

\newcommand{\DD}{\mathbb{D}}
\newcommand{\cD}{{\mathcal{D}}}

\newcommand{\RR}{\mathds{R}}
\newcommand{\NN}{\mathds{N}}
\newcommand{\ZZ}{\mathds{Z}}

\newcommand{\Rd}{\mathds{R}^{d}}
\DeclareMathOperator{\Exp}{\mathbb{E}}
\DeclareMathOperator{\Prob}{\mathbb{P}}
\newcommand{\D}{\mathrm{d}}
\newcommand{\E}{\mathrm{e}}

\newcommand{\Act}{{\mathbb{U}}}
\newcommand{\Uadm}{\mathfrak{U}}
\newcommand{\Usm}{\mathfrak{U}_{\mathrm{SM}}}
\newcommand{\Ussm}{\mathfrak{U}_{\mathrm{SSM}}}

\newcommand{\Ind}{\mathds{1}}   
\newcommand{\Cc}{\mathcal{C}}   

\newcommand{\abs}[1]{\lvert#1\rvert}
\newcommand{\norm}[1]{\lVert#1\rVert}
\newcommand{\babs}[1]{\bigl\lvert#1\bigr\rvert}

\newcommand{\df}{\coloneqq}

\DeclareMathOperator*{\diag}{diag}

\newcommand{\order}{{\mathscr{O}}}
\newcommand{\sorder}{{\mathfrak{o}}}
\newcommand{\grad}{\nabla}

\makeatletter
\DeclareRobustCommand\widecheck[1]{{\mathpalette\@widecheck{#1}}}
\def\@widecheck#1#2{%
    \setbox\z@\hbox{\m@th$#1#2$}%
    \setbox\tw@\hbox{\m@th$#1%
       \widehat{%
          \vrule\@width\z@\@height\ht\z@
          \vrule\@height\z@\@width\wd\z@}$}%
    \dp\tw@-\ht\z@
    \@tempdima\ht\z@ \advance\@tempdima2\ht\tw@ \divide\@tempdima\thr@@
    \setbox\tw@\hbox{%
       \raise\@tempdima\hbox{\scalebox{1}[-1]{\lower\@tempdima\box
\tw@}}}%
    {\ooalign{\box\tw@ \cr \box\z@}}}
\makeatother

\usepackage{accents}
\newlength{\dhatheight}

\let\oldtocsection=\tocsection
\let\oldtocsubsection=\tocsubsection
\let\oldtocsubsubsection=\tocsubsubsection
\renewcommand{\tocsection}[2]{\hspace{0em}\oldtocsection{#1}{#2}}
\renewcommand{\tocsubsection}[2]{\hspace{1em}\oldtocsubsection{#1}{#2}}
\renewcommand{\tocsubsubsection}[2]{\hspace{2em}\oldtocsubsubsection{#1}{#2}}

\newcommand{\ttl}{\Large Optimal control of Markov-modulated\\[5pt] multiclass many-server queues}

\newcommand{\ttls}{Optimal control of Markov-modulated multiclass many-server queues}

\begin{document}

\title[\ttls]{\ttl}

\author[Ari Arapostathis]{Ari Arapostathis$^\dag$}
\address{$^\dag$ Department of Electrical and Computer Engineering\\
The University of Texas at Austin\\
2501 Speedway, EERC 7.824\\
Austin, TX~~78712, USA}
\email{ari@ece.utexas.edu}

\author[Anirban Das]{Anirban Das$^*$}
\address{$^*$ Department of Mathematics\\
Pennsylvania State University,
University Park, PA 16802}
\email{axd353@psu.edu}

\author[Guodong Pang]{Guodong Pang$^\ddag$}
\author[Yi Zheng]{Yi Zheng$^\ddag$}
\address{$^\ddag$ The Harold and Inge Marcus Department of Industrial and
Manufacturing Engineering,
College of Engineering,
Pennsylvania State University,
University Park, PA 16802}
\email{\{gup3,yxz282\}@psu.edu}


\begin{abstract} 
We study multiclass many-server queues for which the arrival, service and abandonment rates are all modulated by a common finite-state Markov process.  We assume that the system operates  in the ``averaged'' Halfin--Whitt regime, which means that it is critically loaded in the average sense,  although not necessarily in each state of the Markov process. 
We show that under any static priority policy, the Markov-modulated diffusion-scaled queueing process is geometrically ergodic. This is accomplished by  employing a solution to an associated Poisson equation in order to construct a suitable Lyapunov function. 
We establish a functional central limit theorem for the diffusion-scaled queueing process and show that the limiting process is a controlled diffusion with piecewise
linear drift and constant covariance matrix. 
We address the infinite-horizon discounted and long-run average (ergodic) optimal control problems and establish asymptotic optimality. 
\end{abstract}


\maketitle



\allowdisplaybreaks

\section{Introduction} 
Queueing networks operating in a random environment have been studied extensively. 
 A functional central limit theorem (FCLT) for  Markov-modulated
infinite-server queues is established in \cite{ABMT}, which shows that the limit process is 
an Ornstein--Uhlenbeck diffusion; 
see also \cite{PZ17,JMTW17} for more recent work.
Scheduling control problems for  Markov-modulated multiclass single-server queueing networks have been addressed in \cite{LQA17,RMH13,BGL14}. In \cite{BGL14}, the authors show that a modified $c\mu$-policy is asymptotically optimal for the infinite horizon discounted problem.
For a single-server queue with only the arrival rates modulated, service rate control problems over a finite and infinite horizon have been studied in \cite{RMH13,LQA17}. 
For multiclass many-server queues without modulation, the infinite-horizon discounted and ergodic control problems have been studied in \cite{AMR04} and \cite{ABP15}, respectively. 

In this paper we address the aforementioned
control problems for Markov-modulated multiclass many-server queues.
We establish the weak convergence of the diffusion-scaled queueing processes, study their stability properties, characterize the optimal solutions via the associated
limiting diffusion control problems, and then prove asymptotic optimality.
Specifically, we assume that the arrival, service and abandonment rates are all modulated by a finite-state Markov process, and that given the state of this process, the arrivals are Poisson, and the service and patient times are exponentially distributed.
The system operates in the ``averaged'' Halfin--Whitt (H--W) regime, namely, it is critically loaded in an average sense, but it may be underloaded or overloaded for a given state of the environment. 
This situation is different from the
standard H--W regime for many-server queues, which requires that the system is critically loaded as the arrival rates and number of servers get large;
see, e.g., \cite{HW81,AMR04,JMTW17,ABP15}.

We first establish a FCLT for the Markov-modulated diffusion-scaled   queueing processes under any admissible scheduling policy (only considering work-conserving and preemptive policies). 
Proper scaling is needed in order to establish  weak convergence of the queueing processes. 
In particular, since the arrival processes are of order $n$, and the switching rates of the background process are assumed to be of order $n^{\alpha}$ for $\alpha>0$, the queueing processes are centered at the `averaged' steady state, which is of order $n$,
and are then scaled down by a factor of an $n^\beta$, with
$\beta \df \max\{\nicefrac{1}{2}, 1-\nicefrac{\alpha}{2}\}$, in the diffusion scale. Thus, when $\alpha\ge 1$,  we have the usual diffusion scaling with
$\beta = \nicefrac{1}{2}$, which is due to the fact that the very fast switching of the environment results in an `averaging' effect for the arrival, service and abandonment processes of the queueing dynamics. 
The limit queueing process is a piecewise Ornstein--Uhlenbeck diffusion process
with a drift and covariance given by the corresponding `averaged' quantities under the stationary distribution of the background process.
When $\alpha=1$, both the variabilities of the queueing and background processes are captured in the covariance matrix, while when $\alpha>1$, only the variabilities of the queueing process is captured. On the other hand,   when $\alpha<1$, the proper diffusion scaling requires $\beta=1-\nicefrac{\alpha}{2}$, for which we obtain a similar  piecewise
Ornstein--Uhlenbeck diffusion process with the covariance matrix capturing the variabilities of the background process only.

The ergodic properties of this class of piecewise linear diffusions (and L\'evy-driven stochastic differential equations) have been studied in \cite{Dieker-Gao,APS},
and these results can be applied directly to our model.
The study of the ergodic properties of the diffusion-scaled processes, however, is challenging.
Ergodicity of switching Markov processes has been an active research subject. 
For switching diffusions, stability has been studied in \cite{Kha12,Mao99,Kha07}. 
However, studies of ergodicity of switching Markov processes are scarce. 
Recently in \cite{BH12, BBMZ15}, some kind of hypoellipticity criterion with
H{\"o}rmander-like bracket conditions is provided to establish geometric
convergence in the total variation distance. 
As pointed out in \cite{Hairer15}, this condition cannot be easily
verified, even for many classes of simple Markov processes with random switching. 
Cloez and Hairer \cite{Hairer15} provided a concrete criterion for exponential ergodicity in situations which do not verify any hypoellipticity assumption (as well as criterion for convergence in terms of Wasserstein distance).
Their proof is based on a coupling argument and a weak form of Harris' theorem. 
It is worth noting that in these studies, the transition rates of the underlying Markov process are unscaled, and therefore, the Markov processes under random switching do not exhibit an `averaging' effect. 
Because of the `averaging' effect in our model, we are able to construct a suitable Lyapunov function to verify the standard Foster-Lyapunov condition in order to prove the geometric exponential ergodicity of the diffusion-scaled queueing processes. 

The technique we employ is much similar in spirit to the approach in \cite{Kha12}
for studying $p$-stability of the switching diffusion processes with rapid switching. 
For diffusions, Khasminskii \cite{Kha12} observes that rapid switching results in some `averaging' effect, and thus if the `averaged' diffusion (modulated parameters are replaced by their averages under the invariant measure of the background process) is stable, then a Lyapunov function can be constructed by using solutions to an associated Poisson equation to verify the Foster-Lyapunov stability condition for the original diffusion process.  
To the best of our knowledge, this approach has not been used to study general fast switching Markov processes.
We employ this technique to the Markov-modulated diffusion-scaled queueing process of the multiclass many-server model.
Ergodicity properties for multiclass Markovian queues have been established in \cite{GS12}; in particular, it is shown that the queueing process is ergodic under any working conserving scheduling policy.
Following the approach in \cite{ABP15}, we show that under a static priority
scheduling policy, the `averaged' diffusion-scaled processes (with the arrival, service and abandonment parameters being replaced by the averaged quantities) are geometrically ergodic (Lemma \ref{L3.1}).  We then construct a Lyapunov function using a Poisson equation associated with the difference of the Markov-modulated diffusion-scaled queueing process and the `averaged' queueing process, and thus verify the Foster-Lyapunov stability criterion for geometric ergodicity.

To study asymptotic optimality for  the discounted problem, we first establish a moment bound for the Markov-modulated diffusion-scaled queueing process, which is uniform under all admissible policies, i.e., work-conserving and non-preemptive polices. 
We then adopt the approach in \cite{AMR04} and construct a sequence of polices which asymptotically converges to the optimal value of the discounted problem for the limiting diffusion process. 
To prove asymptotic optimality for the ergodic problem, it is critical to study  the convergence of the mean empirical measures associated with the Markov-modulated diffusion-scaled queueing processes. 
Unlike the studies in  \cite{ABP15,AP16,AP18}, the Markov modulation makes this work much more challenging. 
For both the lower and upper bounds, we construct an auxiliary (semimartingale) process associated with a diffusion-scaled queueing process and the underlying Markov process.
We then establish the convergence of the mean empirical measure of the auxiliary process, and thus prove that of the Markov-modulated diffusion-scaled queueing processes by establishing their asymptotic equivalence.  
In establishing the upper bound, we adopt the technique developed in \cite{ABP15}.
Using a spatial truncation, we obtain nearly optimal controls for the ergodic problem of our controlled limiting diffusion by fixing a stable Markov control (any constant control) outside a compact set. 
We then map such concatenated controls for the limiting diffusion process to a family of scheduling polices for the auxiliary processes as well as the diffusion-scaled  queueing processes, which also preserve the ergodicity properties.  With these concatenated policies, we are able to prove the upper bound for the value functions.

\subsection{Organization of the paper}
In the next subsection, we summarize the notation used in this paper. 
\cref{S2} contains a detailed description of the Markov-modulated multiclass many-server queueing model.
In \cref{S2.1}, we introduce the scheduling policies considered in this paper.
In \cref{S2.2}, we present the controlled limiting diffusions and the results of weak convergence. 
We state the main results on asymptotic optimality for the discounted and ergodic
problems in \cref{S2.3,S2.4}, respectively. 
In \cref{S3}, we summarize the ergodic properties of the controlled limiting diffusions, and establish the geometric ergodicity of the diffusion-scaled processes.
A characterization of optimal controls for the controlled limiting diffusions, and the proofs of asymptotic optimality are given in \cref{S4}.  
\cref{S5} is devoted to the proofs of \cref{T2.1,L3.1}, while \cref{S6} contain the proofs of some technical results in \cref{S4}.

\subsection{Notation}
We let $\NN$ denote the set of positive integers.
For $k\in\NN$, $\RR^k$ $(\RR^k_+)$ denotes the set of $k$-dimensional real (nonnegative) vectors, and we write $\RR$ $(\RR_+)$ for $k=1$.
For $k\in\NN$, $\ZZ^k_+$ stands for the set of $d$-dimensional nonnegative integer vectors.
For $i= 1,\dots,d$, we let $e_i$ denote the vector in $R^d$ with the $i^{\rm th}$ element equal to $1$ and all
other elements equal to $0$, and define $e = (1,\dots,1)^{\textup{T}}$.
The complement of a set $A\subset\RR^d$ is denoted by $A^c$.
The open ball in $\RR^d$ with center the origin and radius $R$ is denoted by $B_R$.
For $a,b\in\RR$, the minimum (maximum) of $a$ and $b$ is denoted by $a\wedge b$ $(a\vee b)$,  and we let $a^+ \df a\vee0$. 
For $a\in\RR^+$, $\lfloor a \rfloor$ denotes the largest integer not greater than $a$.
Given any vectors $a,b\in\RR^d$, let $\langle a,b \rangle$ denote the inner product.

The Euclidean norm in $\RR^k$ is denoted by $\abs{\,\cdot\,}$.
For $x\in\RR^k$, we let $\norm{x}\df \sum^{k}_{i=1}\abs{x_i}$.  
We use $\Ind$ to denote the indicator function.
We use the notations $\partial_i \df \frac{\partial}{\partial x_i}$ 
and $\partial_{ij} \df \frac{\partial^2}{\partial x_i \partial x_j}$\,.
For a domain $D\subset\RR^d$, the space $\Cc^k(D)$ ($\Cc^{\infty}(D)$) denotes 
the class of functions whose partial derivatives up to order $k$ (of any order) exist and
are continuous, and $\Cc^k_c(D)$ denotes the space of functions in $\Cc^k(D)$ with compact support.
For $D\subset\RR^d$, we let $\Cc^k_b(D)$ denote the set of functions in $\Cc^k(D)$, whose partial derivatives up to order $k$ are continuous and bounded. 
For a nonnegative function $f\in\Cc(\RR^d)$, we use $\order(f)$ to denote the space of function
$g\in\Cc(\RR^d)$ such that
$\sup_{x\in\RR^d} \frac{\abs{g(x)}}{1 + f(x)} \;<\; \infty\,,
$
and we use $\sorder(f)$ to denote the subspace of $\order(f)$ consisting of functions $g\in\Cc(\RR^d)$ such that
  $
  \limsup_{\abs{x}\rightarrow\infty} \frac{\abs{f(x)}}{1 + g(x)} \;=\; 0\,.
  $
The arrows $\rightarrow$ and $\Rightarrow$ are used to denote convergence of real numbers and  convergence in distribution, respectively. 
For any path $X(\cdot)$, $\Delta X(t)$ is used to denote the jump at time $t$.
We use $\langle \cdot \rangle$ to denote the predictable quadratic variation of a square integrable martingale, and use $[\cdot]$ to denote the optional quadratic variation.
We define $\DD \df \DD(\RR_+, \RR)$ as the real-valued function space of all c\'adl\'ag functions on $\RR_+$. We endow the space $\DD$ with the Skorohod $J_1$ topology and
denote this topological space as $(\DD,\cJ)$.
For any complete and separable metric spaces $S_1$ and $S_2$,
we use $S_1\times S_2$ to denote their product space endowed with the maximum metric.
For any complete and separable space $S$, and $k\in\NN$, the $k$-fold product space with
the maximum metric is denoted by $S^k$. 
For $k\in\NN$, $(\DD^k,\cJ)$ denotes the $k$-fold product of $(\DD,\cJ)$ with the product topology. Given a Polish space $E$, $\cP(E)$ denotes the space of probability measures on $E$, endowed with the Prokhorov metric.

\section{The Model and Control Problems}\label{S2}

We consider a sequence of $d$-class Markov-modulated $M/M/n + M$ 
queueing models indexed by $n$.
Define the space of customer classes by $\cI \df \{1,\dots, d\}$.
For $n\in\NN$, let $J^n \df \{J^n(t)\colon t\ge 0\}$ be a continuous-time Markov chain 
with finite state space $\cK \df \{1,\dots,K\}$, with
an irreducible transition rate matrix $n^{\alpha}\cQ$ for some $\alpha>0$.
Thus, $J^n$ has a stationary distribution denoted by 
$\pi = (\pi_1,\cdots,\pi_K)$, for each $n\in\NN$.
We assume that $J^n$ starts from this stationary distribution.

For each $n$ and $i\in\cI$, 
let $A^n_i \df \{A^n_i(t)\colon t\ge0 \}$ denote the arrival process
of  class-$i$ customers in the $n^{\rm th}$ system.
Provided $J^n$ is in state $k$, the arrival rate of class--$i$ customers 
	is defined by $\lambda^n_i(k)\in\RR_+$, and the service time 
	and the patience time are exponentially distributed with rates 
	$\mu^n_i(k)$ and $\gamma^n_i(k)$, respectively. 
Let $A^n$ denote a Markov-modulated
Poisson process, that is, for $t\ge0$, each $n$ and $i\in\cI$, 
\begin{equation*}
A^n_i(t) \;=\; A_{*,i}\left(\int_{0}^{t}\lambda^n_i(J^n(s))\,\D{s}\right)\,,
\end{equation*}
where $A_{*,i}$'s are mutually independent unit-rate Poisson processes.
In each class, we also assume that the arrival, service and abandonment processes are 
mutually independent. 

Let $X^n$, $Q^n$ and $Z^n$ denote the $d$-dimensional processes counting the number of customers of each class in the $n^{\rm th}$ system, in queue and in service, respectively,
and the following constraints are satisfied: for $t\ge0$ and $i\in\cI$,
\begin{equation}\label{con-XQZ}
\begin{aligned}
&X_i^n(t) \;=\; Q_i^n(t) + Z_i^n(t) \,, \\
&Q_i^n(t) \;\ge\; 0\,, \quad Z_i^n(t) 
\;\ge\; 0 \quad \text{and} \quad \langle{e},Z^n(t)\rangle \;\le\; n\,.
\end{aligned}
\end{equation}     
Then, we have the following dynamic equation: for $t\ge0$, $n\in\NN$ and $i\in\cI$,
\begin{equation}\label{E2.2}
X^n_i(t) \;=\; X_i^{n}(0) + A^n_i(t) - S_i^n(t) - R_i^n(t)  \,, 
\end{equation} 
where 
$$S_i^n(t)\;\df\;S_{*,i}\left(\int_{0}^{t}\mu_i^n(J^n(s))Z_i^n(s)\,\D{s}\right)\,,
\quad R_i^n(t)\;\df\; R_{*,i}\left(\int_{0}^{t}\gamma_i^n(J^n(s))Q_i^n(s)\,\D{s}\right)\,,$$
and $\{S_{*,i}, R_{*,i}\colon i\in\cI \}$ 
are mutually independent unit-rate Poisson processes.

\begin{assumption}\label{A2.1}
As $n\rightarrow\infty$, for $i\in\cI$ and $k\in\cK$, 
\begin{align*}
\begin{aligned}
n^{-1}\lambda_i^n(k) \;\rightarrow\; \lambda_i(k)\;>\;0\,, 
\quad \mu_i^n(k) \;\rightarrow\;\mu_i(k) \;>\;0 \,,
\quad \gamma^n_i(k) \;\rightarrow \gamma_i(k) \;>\; 0\,, \\
n^{-\beta}\left(\lambda^n_i(k)- n\lambda_i(k)\right)
\;\rightarrow\; \Hat{\lambda}_i(k)\quad 
and \quad  n^{1-\beta}(\mu^n_i(k) - \mu_i(k))\;\rightarrow\; \Hat{\mu}_i(k)\,,
\end{aligned}
\end{align*}
where 
$$\beta \;\df\; \max\{\nicefrac{1}{2},1-\nicefrac{\alpha}{2}\}\,.$$
\end{assumption}
For $i\in\cI$ and $n\in \NN$, we define
\begin{align*}
& \lambda^{\pi}_i \;\df\;  \sum_{k\in\cK} \pi_k\lambda_i(k)\, \quad
\mu^{\pi}_i \;\df\; \sum_{k\in\cK} \pi_k\mu_i(k)\,,\quad
\gamma^{\pi}_i \;\df\; \sum_{k\in\cK} \pi_k\gamma_i(k)\,, \\
&\Bar{\lambda}^{n}_i \;\df\;  \sum_{k\in\cK} \pi_k\lambda^n_i(k)\, \quad
\Bar{\mu}^{n}_i \;\df\; \sum_{k\in\cK} \pi_k\mu^n_i(k)\,,\quad
\Bar{\gamma}^{n}_i \;\df\; \sum_{k\in\cK} \pi_k\gamma^n_i(k)\,,
\end{align*} 
and 
$$\rho_i \;\df\; \nicefrac{\lambda^{\pi}_i}{\mu^{\pi}_i}, \quad \rho^{n} \df n^{-1}\sum_{i\in\cI}\nicefrac{\Bar{\lambda}^{n}_i}{\Bar{\mu}^{n}_i}\,. $$

\begin{assumption}\label{A2.2}
The system is critically loaded, that is, $\sum_{i\in\cI} \rho_i = 1$.
\end{assumption}
 
Under \cref{A2.1,A2.2}, we have 
\begin{equation*} 
 n^{1-\beta}(1 - \rho^{n}) 
  \;=\; \sum_{i\in\cI}
 \frac{ n^{-\beta}(n\Bar{\mu}^{n}_i - n\mu_i^\pi)\rho_i 
 	- n^{-\beta}(\Bar{\lambda}^{n}_i - n\lambda^{\pi}_i)}{\Bar{\mu}^{n}_i} 
\,\xrightarrow[n\to\infty]{}\, \sum_{i\in\cI} \frac{\rho_i\Hat{\mu}^{\pi}_i - \Hat{\lambda}^{\pi}_i}{\mu^{\pi}_i}\,,
\end{equation*}
with 
$$\Hat{\lambda}^{\pi}_i \;\df\; \sum_{k\in\cK}\pi_k\Hat{\lambda}_i(k), \quad \Hat{\mu}^{\pi}_i \;\df\; \sum_{k\in\cK}\pi_k\Hat{\mu}_i(k)\,.
 $$ 
 
\cref{A2.1,A2.2} are in effect throughout the paper, without further mention.
A model satisfying these assumptions is said to be in the ``averaged''  H--W regime.


Let $\Bar{X}^n$, $\Bar{Z}^n$,  $\Bar{Q}^n$, 
$\Hat{X}^n$, $\Hat{Z}^n$ and $\Hat{Q}^n$ denote the $d$-dimensional processes satisfying
\begin{align*}
\Bar{X}^n_i &\;=\; n^{-1}X^n_i\,, \quad \Bar{Z}^n_i \;=\; n^{-1}Z^n_i\,,\quad
 \Bar{Q}^n_i\;=\; n^{-1}Q^n_i\,,
\\
\Hat{X}^n_i &\;=\; n^{-\beta}(X^n_i - \rho_in)\,,\quad  
\Hat{Z}^n_i \;=\; n^{-\beta}(Z^n_i - \rho_in)
\quad \text{and} \quad \Hat{Q}^n_i  \;=\; n^{-\beta}Q^n_i
\end{align*}
for $i\in\cI$. Then, for $t\ge0$ and $i\in\cI$, $\Hat{X}_i^n(t)$ can be written as  
\begin{align}\label{HatX}
\begin{aligned}
\Hat{X}_i^n(t) &\;=\; \Hat{X}_i^n(0) + \Hat{\ell}_i^n(t)  + \Hat{L}_i^n(t)+ \Hat{A}_i^n(t) - \Hat{S}_i^n(t)
- \Hat{R}_i^n(t) \\
&\qquad\qquad\qquad - \int_{0}^{t}\mu_i^n(J^n(s))\Hat{Z}^n_i(s)\,\D{s}
- \int_{0}^{t}\gamma_i^n(J^n(s))\Hat{Q}^n_i(s)\,\D{s}\,,
\end{aligned}
\end{align}
where
\begin{align*}
\begin{aligned}
\Hat{\ell}_i^n(t) &\;\df\; n^{-\beta}\sum_{k\in\cK}\Big(\big(\la^n_i(k)-n\la_i(k)\big) 
 - n\rho_i\big(\mu^n_i(k) - \mu_i(k)\big)\Big)
 \int_{0}^{t}\Ind(J^n(s)=k)\,\D{s}\,,\\
\Hat{L}^n_i(t) &\;\df\;  n^{1-\beta}\int_{0}^{t}(\lambda_i(J^n(s))-\lambda_i^\pi)\,\D{s}  
- n^{1-\beta}\rho_i\int_{0}^{t}(\mu_i(J^n(s))-\mu_i^\pi)\,\D{s} \,, \\
\Hat{A}^n_i(t) &\;\df\; 
n^{-\beta}\left(A^n(t) - \int_{0}^t\lambda_i^n(J^n(s))\,\D{s} \right)\,, \\
\Hat{S}_i^n(t) &\;\df\; 
n^{-\beta}\left(S_i^n(t) - \int_0^{t}\mu_i^n(J^n(s))Z^n_i(s)\,\D{s}\right)\,,\\
\Hat{R}_i^n(t) &\;\df\; 
n^{-\beta}\left(R_i^n(t) - \int_0^{t}\gamma_i^n(J^n(s))Q^n_i(s)\,\D{s}\right)\,. 
\end{aligned}
\end{align*}

Define the random processes $\Hat{Y}^n = (\Hat{Y}^n_1,\dots,\Hat{Y}^n_d)'$ 
and $\Hat{W}^n =(\Hat{W}^n_1,\dots,\Hat{W}^n_d)'$ by 
\begin{equation*}
\Hat{Y}^n_i(t) \;\df\; \Hat{\ell}^n_i(t) - \int_{0}^{t}\mu_i^n(J^n(s))\Hat{Z}^n_i(s)\,\D{s}
- \int_{0}^{t}\gamma_i^n(J^n(s))\Hat{Q}^n_i(s)\,\D{s}
\end{equation*}
for $i \in \cI$ and $t\ge 0$, and 
\begin{equation*}
\Hat{W}^n \;\df\; \Hat{L}^n + \Hat{A}^n  - \Hat{S}^n - \Hat{R}^n\,, 
\end{equation*}
respectively. 
Then, \cref{HatX} can be written as
\begin{equation*}
\Hat{X}^n(t) \;=\; \Hat{X}^n(0) + \Hat{Y}^n(t) + \Hat{W}^n(t)\,, \quad t \ge 0\,. 
\end{equation*}
Throughout the paper, 
we assume that $\{\Hat{X}^n(0)\colon n\in\NN\}$ are deterministic. 

\subsection{Scheduling policies}\label{S2.1}

Let $\tau^n(t) \df \inf\{r\ge t\colon J^n(r) = 1 \}$ for $t\ge 0$.
We define the following filtrations: for $t \ge 0, r \ge 0,$
\begin{align*}
\mathcal{F}_t^n &\;\df\; \sigma\{A^n_i(s),S^n_i(s),R^n_i(s), Q_i^n(s), Z_i^n(s), X_i^n(s), J^n(s)  \colon i\in\cI,s\le t  \}\vee\mathcal{N} \,, \\
\mathcal{G}_{t,r}^n &\;\df\;  \sigma\{A^n_i(\tau^n(t) + r) - A^n_i(\tau^n(t)),S^n_i(\tau^n(t) + r) - S^n_i(\tau^n(t)), \\
 &\quad\quad\quad\quad\quad\quad  R^n_i(\tau^n(t) + r) - R^n_i(\tau^n(t)),\Hat{L}^n_i(\tau^n(t) + r) - \Hat{L}^n_i(\tau^n(t))    \colon i\in\cI  \}\vee\mathcal{N}\,,
\end{align*}
where $\mathcal{N}$ is a collection of $\Prob$-null sets.

\begin{definition}\label{D2.1}
We say a scheduling policy $Z^n$ is admissible, 
if it satisfies following conditions.
\begin{itemize}
\item[(\textup{i})] Preemptive: a server can stop serving a class of customer to serve some other class of customers at any time, and resume the original service at a later time.
\item[(\textup{ii})] Work-conserving: for each $t\ge 0$, $\langle e,Z^n(t)\rangle = \langle e,X^n(t)\rangle\wedge n$.
\item[\textup{(iii)}] Non-anticipative: for $t\ge0$ and $r\ge0$,
\begin{itemize}
	\item[(a)] $Z^n(t)$ is adapted to $\mathcal{F}^n_t$.
	\item[(b)] $\mathcal{F}_t^n$ and $\mathcal{G}_{t,r}^n$ are independent.
\end{itemize}

\end{itemize}
\end{definition}

We only consider  admissible scheduling policies. 
Given an admissible scheduling policy $Z^n$, the process $X^n$ 
in \eqref{E2.2} is well defined, and we say that it is governed by the scheduling policy $Z^n$.  
Abusing the terminology we equivalently also say that
$\Hat{X}^n$ is governed by the scheduling policy 
$\Hat{Z}^n$. 

	Define the set 
	\begin{equation*}
	\Act \;\df\; \{u\in\RR^d_{+}: \langle e,u \rangle = 1\}\,,
	\end{equation*}
	It is often useful to re-parametrize and  replace the scheduling policy $Z^n$ with a new scheduling policy $\Hat{U}^n$ defined as follows.
	Given a process $X^n$ defined using \eqref{E2.2} and an admissible scheduling policy $Z^n$, for $t\ge0$,
	define
	\[\Hat{U}^n(t) \;\df\;
	\begin{cases} 
	\frac{Z^n(t) - X^n(t)}{n-\langle e,X^n(t) \rangle} \quad &\text{for}\;  \langle e,X^n(t) \rangle>n\,,\\
	e_d  \quad &\text{for}\;   \langle e,X^n(t) \rangle \le n\,.
	\end{cases}
	\]
	$\Hat{U}^n(t)$ is a $\Act$--valued process, representing the  proportion of class-$i$ customers in the queue.  Any process $\Hat{U}^n$ defined as above by using some admissible scheduling policy $Z^n$ is called an admissible proportions-scheduling policy. 
	The set of all such admissible proportions-scheduling polices $\Hat{U}^n$ 
	is denoted by $\widehat{\Uadm}^n$. Abusing terminology we replace the term admissible  proportions-scheduling policy with admissible scheduling policy.


\subsection{The limiting controlled diffusion}\label{S2.2}

By the equation (4) in \cite{ABMT} and \cref{A2.1}, we have
\begin{equation}\label{E2.7}
\Hat{\ell}^n(t) \;\rightarrow\; \ell t \quad \text{a.s.} \quad \text{as} \quad n \ra\infty, 
\end{equation}
where $\Hat{\ell}^n \df (\Hat{\ell}_1,\dots,\Hat{\ell}_d)'$, $\ell \df (\ell_1,\dots,\ell_d)'$ and 
$\ell_i \df \Hat{\lambda}_i^\pi - \rho_i \Hat{\mu}_i^\pi$.
Let $\pi$ be the stationary distribution of $J^n$, that is, $\pi' \cQ =0$ and $\pi'e=1$. (Note that scaling $\cQ$ does not change the stationary distribution.)
By Proposition 3.2 in \cite{ABMT}, we obtain
\begin{equation}\label{E2.8}
\Hat{L}^n \;\Rightarrow\; 
 \sigma^L_\alpha \widetilde{W} \quad \text{in} \quad (\DD^d, \cJ)\,, \quad \text{as} \quad n\rightarrow\infty
 \,,
\end{equation}
where $\Hat{L}^n\df(\Hat{L}^n_1,\dots,\Hat{L}^n_d)'$,
$\widetilde{W}$ is a zero-drift standard $d$-dimensional Wiener process, 
and if $\alpha>1$, then $\sigma^L_\alpha =0$, while if $\alpha \le 1$,
then
$\sigma^L_\alpha$ satisfies
$(\sigma^L_\alpha)'\sigma^L_\alpha =  \Theta = [\theta_{ij}]$, with 
\begin{equation*}
\theta_{ij}\;\df\; 2\sum_{l\in\cK}\sum_{k\in\cK} \big(\la_i(k) - \rho_i \mu_i(k)\big) \big(\la_j(l) - \rho_j \mu_j(l)\big)\pi_k\Upsilon_{kl} 
\end{equation*}
for $i,j\in\cI$, and $\Upsilon_{kl}\df\int_{0}^{+\infty}\left(P_{kl}(t)-\pi_k\right)\D{t}$ 
with $P_{kl}(t) = [\E^{\cQ t}]_{kl}$, 
that is, 
$\Upsilon = (\Pi - \cQ)^{-1} - \Pi$.
Here, $\Pi$ denotes the matrix whose rows are equal to the vector $\pi$. 

The proof of the following result is in \cref{S5}.

\begin{theorem} \label{T2.1}
Under \cref{A2.1,A2.2}, and assuming that $\Hat{X}^n(0)$ is uniformly bounded, the following results hold.
\begin{itemize}
\item[(\textup{i})]
As $n\rightarrow \infty$,
\begin{equation*}
(\Bar{Z}^n, \Bar{Q}^n) \;\Rightarrow\; (\rho,0) 
\quad \text{in} \quad (\DD^d,\cJ)^2\,,
\end{equation*}
where $\rho = (\rho_1,\dots,\rho_d)$.	
		
\item[(\textup{ii})]
As $n\rightarrow\infty$,
\begin{equation*}
\Hat{W}^n \;\Rightarrow\; \Hat{W} \quad \text{in} \quad (\DD^d, \cJ)\,,
\end{equation*}
where $\Hat{W}$ is a $d$-dimensional Brownian motion with a covariance coefficient matrix $\sigma'_{\alpha}\sigma_\alpha$ defined by 
$$
\sigma'_\alpha \sigma_\alpha = \begin{cases}
\Lambda^2, & \alpha >1, \\
\Lambda^2 + \Theta, & \alpha=1, \\
\Theta, & \alpha<1,
\end{cases} 
$$
and  $\Lambda\df \diag(\sqrt{2\lambda_1^{\pi}},\dots,\sqrt{2\lambda_d^{\pi}})$.

\item[(\textup{iii})]
$(\Hat{X}^n, \Hat{W}^n, \Hat{Y}^n)$ is tight in $(D^d,\cJ)^3$.

\item[(\textup{iv})] 
Provided that $\Hat{U}^n$ is tight, any limit $\Hat{X}$ of $\Hat{X}^n$ is a unique strong solution of  
\begin{equation}\label{ET2.1A}
\D\Hat{X}(t) \,=\, b\big(\Hat{X}(t),\Hat{U}(t)\big)\D{t}  + \D\hat{W}(t)\,,
\end{equation}
where $\hat{X}(0)=x$, $x\in\RR^d$, is a limit of  $\Hat{X}^n(0)$,  $\Hat{U}$ is a limit of $\Hat{U}^n$,
and $b\colon\RR^d\times\Act\mapsto\RR^d$ satisfies  
\begin{equation*}
b(x,u) \;=\; \ell - M(x - \langle e,x \rangle^+u) - \Gamma\langle e,x \rangle^+u
\end{equation*}
with $M = \diag(\mu_1^\pi,\dots,\mu_d^{\pi})$ and 
$\Gamma = \diag(\gamma_1^\pi,\dots,\gamma_d^{\pi})$. 
Furthermore,
$\Hat{U}$ is non-anticipative, that is, for $s\le t$, $\Hat{W}(t) - \Hat{W}(s)$ is independent of 
$$\mathcal{F}_s \;\df\; \sigma(\Hat{U}(r), \Hat{W}(r) \colon r\le s)\vee \mathcal{N}\,,$$
where $\mathcal{N}$ is the collection of $\Prob$-null sets.
\end{itemize}
\end{theorem}

\subsection{The discounted cost problem}\label{S2.3}

Let $\widetilde{\rc}\colon \RR_+^d \mapsto \RR_+$ take the form 
\begin{equation}\label{E2.3A}
	\widetilde{\rc}(x) = c \abs{x}^m, 
\end{equation}
for some $c>0$ and $m \ge 1$. 
Define  the running cost function
 $\rc\colon \RR^d \times \Act \mapsto \RR_+$  by 
$$\rc(x,u)  \;\df\; \widetilde{\rc}(\langle e,x \rangle^+u).$$

\begin{remark}
In place of \eqref{E2.3A} one may merely stipulate that $\widetilde\rc(x)$ is  a
locally H\"older continuous function such that 
	\begin{equation} \label{A2.3-E1}
 c_1 	 \abs{x}^{\underline{m}}  \;\le\;  \widetilde\rc(x) \;\le\; c_2(1+ \abs{x}^{\overline{m}}) 
	\end{equation}
	for some constants $1\le\underline{m}\le\overline{m}$. See, e.g., Remark 3.1 in \cite{AP18}.
 For the discounted problem, the lower bound in \eqref{A2.3-E1} is not required.
\end{remark}

For each $n$ and $\vartheta>0 $, given $\Hat{X}^n(0)$, the $\vartheta$-discounted problem can be written as
$$\Hat{V}^n_{\vartheta}(\Hat{X}^n(0)) \;\df\; 
\inf_{\Hat{U}^n\in\widehat{\Uadm}^n}\,\mathfrak{J}^n_{\vartheta}(\Hat{X}^n(0),\Hat{U}^n) \,,
$$
and 
$$ \mathfrak{J}^n_{\vartheta}(\Hat{X}^n(0), \Hat{U}^n) \;\df\;   
\Exp\left[\int_0^{\infty}\E^{-\vartheta s}\rc(\Hat{X}^n(s),\Hat{U}^n(s))\,\D{s}\right] 
\,.$$

Let $\Uadm$ denote the set of all admissible controls for the limiting diffusion in \cref{ET2.1A}. 
The $\vartheta$-discounted cost criterion for the limiting controlled diffusion is defined by 
$$
\mathfrak{J}_{\vartheta}(x, \Hat{U}) \;\df\;   
\Exp_x^{\Hat{U}}\left[\int_0^{\infty}\E^{-\vartheta s}\rc(\Hat{X}(s),\Hat{U}(s))\,\D{s}\right]
\,, 
$$ 
for $\Hat{U}\in\Uadm$, and the $\vartheta$-discounted problem is
\begin{equation*}
\Hat{V}_{\vartheta}(x) \;\df\; \inf_{\Hat{U}\in\Uadm}\mathfrak{J}_{\vartheta}(x,\Hat{U}) \,. 
\end{equation*}

In the study of the discounted and ergodic control problems, we assume that
$\Hat{X}^n(0)\rightarrow x\in\RR^d$ as $n\rightarrow\infty$,
and, as mentioned earlier, we impose the conditions in \cref{A2.1,A2.2}.

\begin{theorem}\label{T2.2}
For any $\vartheta > 0$, it holds that
$$
 \lim_{n\rightarrow\infty} \Hat{V}^n_{
 	\vartheta}(\Hat{X}^n(0)) \;=\; \Hat{V}_{\vartheta}(x)\,.
$$
\end{theorem}


\subsection{The ergodic control problem}\label{S2.4}
Given $\Hat{X}^n(0)$, define the ergodic cost associated with $\Hat{X}^n$ and $\Hat{U}^n$ by
$$ \mathfrak{J}^n (\Hat{X}^n(0),\Hat{U}^n) \;\df\;
\limsup_{T\rightarrow\infty}\frac{1}{T}\Exp\left[\int_{0}^{T}\rc(\Hat{X}^n(s),\Hat{U}^n(s))\,\D{s}\right]\,, $$
and the associated ergodic control problem by 
$$ \Hat{V}^n(\Hat{X}^n(0)) \;\df\; \inf_{\Hat{U}^n\in\widehat{\Uadm}^n} \mathfrak{J}^n(\Hat{X}^n(0), \Hat{U}^n)\,$$
Analogously, we define the ergodic cost associated with the limiting controlled process $\Hat{X}$ in
\cref{ET2.1A} by 
$$ \mathfrak{J}(x,\Hat{U}) \;\df\; \limsup_{T\rightarrow\infty}\frac{1}{T}\Exp\left[\int_{0}^{T}\rc(\Hat{X}(s),\Hat{U}(s))\,\D{s} \right]\,,$$
and the ergodic control problem by
$$\varrho_*(x)\;\df\;\inf_{U\in\Uadm} \mathfrak{J}(x,U)\,.$$
The value $\varrho_*(x)$ is independent of $x$. 
As we show in \cref{T4.2}, the infimum is realized with a stationary Markov control and
$\varrho_*(x) = \varrho_*$.

The asymptotic optimality of the value functions is stated below (proof in \cref{S4.3,S4.4}). 
\begin{theorem}
It holds that
	$$\lim_{n\rightarrow\infty}\Hat{V}^n(\Hat{X}^n(0))\;=\;\varrho_*(x)\,. $$ 
\end{theorem}

For $u\in\Uadm$, let $\cL_u\colon\Cc^2(\RR^d) \mapsto \Cc^2(\RR^d\times\Act)$ be the generator of 
$\Hat{X}$ in \cref{ET2.1A} satisfying 
\begin{equation}\label{def-Lu}
\cL_uf(x) \;=\; \sum_{i\in\cI}b_i(x,u)\partial_if(x) + \sum_{i,j\in\cI}a_{ij}\partial_{ij}f(x)\,,
\end{equation}
with $a_{ii} \df \Ind(\alpha\ge1)\lambda^{\pi}_i + \frac{1}{2}\Ind(\alpha\le1)\theta_{ii}$, 
and $a_{ij} \df \frac{1}{2}\Ind(\alpha\le1)\theta_{ij}$ for $i\neq j$.
We denote by $\Usm$, $\Ussm$ and $\eom$, the sets of stationary controls, stable controls and ergodic occupation measures, respectively.
We extend the definition of $b$ and $\rc$ 
by using the relaxed control framework (see, for example, Section 2.3 in \cite{ABG12}).
Without changing the notation, for $v\in\Usm$, we replace $b_i$ by 
$$
	b_i(x,v(x)) \;=\; \int_{\Act}b_i(x,u)\,v(du\,|\,x)\,, \quad \text{for}\;i\in\cI\,,
$$
where $v(du\,|\,x)$ denotes a Borel measurable kernel on $\Act$ given $x$,
and replace $\rc$ analogously. 
If a control is a measurable map from $\RR^d$ to $\Act$, we say it is a precise control.  
\section{Ergodic Properties}\label{S3}

The limiting diffusion belongs to the class of piecewise linear diffusions studied in \cite{Dieker-Gao}. Applying Theorem 3 in \cite{Dieker-Gao}, we deduce that 
the limiting process $\hat{X}$ with abandonment in \cref{ET2.1A} is exponentially ergodic under a constant control $\bar{u}=e_d=(0,\dots,0, 1)'$. 
By Theorem 3.5 in \cite{APS}, the limiting process $\hat{X}$ in \cref{ET2.1A} is exponentially ergodic  under any constant control.
 (Note that in \cite[Theorem 2]{Dieker-Gao}, only positive recurrence is shown under the control $\bar{u}=e_d.$)
We summarize the ergodicity properties of the limiting controlled process $\hat{X}$ in the following proposition.

\begin{proposition}
The controlled diffusion $\hat{X}$ in \eqref{ET2.1A} is exponentially ergodic under any constant control $u \in \Act$. 
\end{proposition}

\begin{remark}
As a consequence of the proposition, if $\tilde{v}$ is  a stationary Markov control
which is constant on the complement of some compact set, then  the controlled diffusion $\Hat{X}$ in \cref{ET2.1A} is exponentially ergodic under this control. For the diffusion-scaled process $\hat{X}^n$, we first prove exponential ergodicity
 under a static priority scheduling policy in Lemma \ref{L3.1}. 
It then follows that any stationary Markov scheduling policy, which
agrees with this static priority policy outside a compact set, 
is geometrically ergodic. 
We remark here that exponential ergodicity of the diffusion-scaled process under any stationary Markov scheduling policy is an open problem
(compare with the study of ergodicity for the standard `V' network in \cite{GS12}). 
\end{remark}

We next focus on the ergodicity properties of the diffusion-scaled process $\hat{X}^n$. 

\begin{definition}\label{D3.1}
For each $n\in\NN$, let $\tilde{z}^n = \tilde{z}^n(x)$, for $x\in\ZZ^d_+$,
denote the scheduling policy defined by
$$
\tilde{z}^n_i(x) \;\df\; x_i \wedge \left(n - \sum_{i'=1}^{i-1}x_{i'} \right)^+ 
\quad \text{for} \quad i\in\cI \,.
$$ 
\end{definition}
By using the balance equation $x_i = \tilde{z}^n_i(x) + \tilde{q}^n_i(x)$ and \cref{D3.1}, we obtain for $x\in\ZZ^d_+$ and $i\in\cI$, that
$$\tilde{q}^n_i(x) \;=\; \left[x_i - \left(n - \sum_{i'=1}^{i-1}x_{i'}\right)^+ \right]^+ \,.$$

\begin{definition}\label{D3.2}
For $x\in\RR^d$, define
\begin{align*}
\tilde{x}^n(x) \;\df\; \left(x_1 - \rho_1n,\dots,x_d - \rho_dn\right)'\,,
\quad 
\Hat{x}^n(x) \;\df\; n^{-\beta}\tilde{x}^n(x)\,,
\end{align*}
$\fX^n \df \{\Hat{x}^n(x)\colon x\in\ZZ^d_+\}$, and $\tilde{\fX}^n \df \{\Hat{x}^n(x)\colon x\in\fA^n \}$, 
with
$$\fA^n \,\df\, \bigl\{x\in\RR^d_+ \colon \norm{x-\rho n}\;\le\;c_0n^{\beta}\bigr \}  $$
for some positive constant $c_0$.
\end{definition}

\begin{definition}\label{D3.3}
	Denote the infinitesimal generator of the ``average'' process by 
	\begin{multline*}
	\Bar{\Lg}^{z^n}_nf(x)\;\df\; \sum_{i\in\cI}\Bar{\lambda}_i^n \big(f(x+ e_i) - f(x)\big)
	+ \sum_{i\in\cI}\Bar{\mu}_i^nz^n_i(x)\big(f(x - e_i) - f(x)\big) \\
	+ \sum_{i\in\cI}\Bar{\gamma}_i^nq^n_i(x,z)\big(f(x - e_i) - f(x)\big)
	\end{multline*}
	for $f\in\Cc_b(\RR^d)$.
\end{definition}

\begin{lemma}\label{L3.1} 
Let $\tilde{z}^n$ be the scheduling policy in \cref{D3.1}.
Then for any even integer $m\ge2$, there exist a positive vector $\xi$, 
positive constants  $C_1$ and $C_2$, and $n_0\in\NN$, such that the functions $f_n$, $n\in\NN$, defined by
\begin{equation}\label{EL3.1A}
f_n(x) \;\df\; \sum_{i\in\cI}\xi_i\abs{x_i - \rho_in}^m\,, \quad \forall\;x\in\ZZ^d_+ \,,
\end{equation}
satisfy
\begin{equation*}
\Bar{\Lg}^{\tilde{z}^n}_nf_n(x) \;\le\; C_1n^{m\beta} - C_2f_n(x)\,,
\quad \forall x\in\ZZ^n_+\,, \quad \forall n\ge n_0\,.
\end{equation*}
\end{lemma}
For a proof of \cref{L3.1}, see  \cref{S5}. 
This lemma shows that, under the static priority policy $\tilde{z}^n$, 
the ``average'' process is geometrically ergodic.


\begin{definition}\label{D3.4}
Under a stationary Markov policy $z^n = z^n(x)$, 
the infinitesimal generator of $(X^n(t),J^n(t))$ is defined by
\begin{equation*}
\widetilde{\Lg}^{z^n}_nf(x,k) \;\df\; 
\Lg^{z^n}_{n,k}f(x,k) + \sum_{k'\in\cK} n^{\alpha}q_{kk'}\big(f(x,k') - f(x,k)\big)\,,
\end{equation*}
for $f\in\Cc_b(\RR^d\times\cK)$, where
\begin{multline*}
\Lg^{z^n}_{n,k}f(x,k)\;\df\; \sum_{i\in\cI}\lambda_i^n(k) \big(f(x+ e_i,k) - f(x,k)\big)
+ \sum_{i\in\cI}\mu_i^n(k)z^n_i(x)\big(f(x - e_i,k) - f(x,k)\big) \\
+ \sum_{i\in\cI}\gamma_i^n(k)q^n_i(x,z)\big(f(x - e_i,k) - f(x,k)\big)\,.
\end{multline*}
Let $\Delta{\lambda}^n_i(k) \;\df\;  \Bar{\lambda}^n_i - \lambda^n_i(k)$ for $i\in\cI$ and 
$k\in\cK$, and define $\Delta{\mu}^n_i$ and $\Delta{\gamma}^n_i$, analogously.
Let $\Delta\Lg^{z^n}_{n,k}\colon\Cc_b(\RR^d\times\cK) \mapsto \Cc_b(\RR^d\times\cK)$ be the operator defined by
\begin{multline*}
\Delta\Lg^{z^n}_{n,k}f(x,k)\;\df\; \sum_{i\in\cI}\Delta\lambda_i^n(k) \big(f(x+ e_i,k) - f(x,k)\big)
+ \sum_{i\in\cI}\Delta\mu_i^n(k)z^n_i(x)\big(f(x - e_i,k) - f(x,k)\big) \\
+ \sum_{i\in\cI}\Delta\gamma_i^n(k)q^n_i(x,z)\big(f(x - e_i,k) - f(x,k)\big)\,.
\end{multline*}
\end{definition}
Define the embedding $\mathfrak{M}\colon \Cc_b(\RR^d) \hookrightarrow \Cc_b(\RR^d \times \cK)$
by $\mathfrak{M}(f) = \tilde{f}$, where
$
	\tilde{f}(\cdot,k) = f(\cdot)  
$
for all $k\in\cK$.
It is easy to see, by \cref{D3.3,D3.4}, that for all $f\in\Cc_b(\RR^d)$, $\tilde{f} = \mathfrak{M}(f)$, and $k\in\cK$, we have
\begin{equation}\label{E3.4A}
\Bar{\Lg}^{z^n}_n{f}(x) \;=\;  \Lg^{z^n}_{n,k}\tilde{f}(x,k) 
+ \Delta\Lg^{z^n}_{n,k}\tilde{f}(x,k).
\end{equation}

Abusing the notation, we can identify $\tilde{f} = \mathfrak{M}(f)$ with $f$,
 and thus \cref{E3.4A} can be written as
$$
\Bar{\Lg}^{z^n}_nf(x) \;=\;  \Lg^{z^n}_{n,k}f(x) 
+ \Delta\Lg^{z^n}_{n,k}f(x)\,.
$$
\begin{lemma}\label{L3.2}
Let $f_n(x)$ be the function defined in \cref{EL3.1A}, and $z^n$ be any stationary Markov policy. 
There exists a function $g_n[f_n]\in \Cc(\RR^d \times \cK)$, satisfying
\begin{equation}\label{EL3.2A}
g_n[f_n](x,k) = \frac{1}{n^{\alpha}} \sum_{k'\in\cK}c_{kk'}\Delta\Lg^{z^n}_{n,k'}f_n(x)\,, \quad \forall\;(x,k)\in\RR^d\times\cK
\end{equation}
with some constants $c_{kk'}$ independent of $n$, and
$$
 \sum_{k'\in\cK} n^{\alpha}q_{kk'}\big(g_n[f_n](x,k') - g_n[f_n](x,k)\big) \;=\; \Delta\Lg^{z^n}_{n,k}f_n(x)\,, \quad \forall\;(x,k)\in\RR^d\times\cK\,.
$$
As a consequence, we have, for fixed $\alpha > 0$ and each $n\in\NN$,
\begin{equation}\label{EL3.2B}
\big|g_n[f_n](x,k)\big| \;\le\; C_3(1 + n^{m(1-\alpha)}) + \epsilon_n f_n(x)\,, \quad \forall (x,k)\in\RR^d\times\cK\,,
\end{equation}
where $C_3$ is some positive constant, and $\epsilon_n>0$ can be chosen arbitrarily small for large enough $n$. 
\end{lemma}

\begin{proof}
	The existence of $g_n[f_n](x,k)$ directly follows from the Fredholm alternative. The version applicable here may be found in \cite{Kha12}.

	For $k\in\cK$, we observe that
\begin{align*}
\frac{\abs{\Delta\Lg^{z^n}_{n,k}f_n(x)}}{n^{\alpha}} \;\le&\; 
\frac{1}{n^{\alpha}}\sum_{i \in \cI}\xi_i
\babs{\Delta{\lambda}_i^n(k) - \Delta{\mu}_i^n(k)z^n_i - \Delta{\gamma}_i^n(k)q^n_i}
\babs{m(\tilde{x}^n_i)^{m-1} + \order\bigl(\abs{\tilde{x}^n_i}^{m-2}\bigr)} \\
\;\le&\;\frac{1}{n^{\alpha}}\sum_{i \in \cI}\xi_i
\Big(\abs{\Delta{\lambda}_i^n(k)} + \abs{\Delta{\mu}_i^n(k)}x_i + \abs{\Delta{\gamma}_i^n(k)}x_i \Big)
\babs{m(\tilde{x}^n_i)^{m-1} + \order\bigl(\abs{\tilde{x}^n_i}^{m-2}\bigr)} \\
\;\le&\; C_4\bigl(1 + n^{m(1-\alpha)}\bigr) + \epsilon_n f_n(x)\,,
\end{align*}
where $C_4$ is some positive constant, and 
the last inequality follows by using $ x_i  = \tilde{x}_i^n + n\rho_i$, \cref{A2.1}, and
following inequalities with sufficiently small $\epsilon > 0$:
\begin{align}\label{PL3.2A}
\begin{aligned}
n^{-\alpha}\abs{\Delta{\lambda}_i^n(k)}\abs{\tilde{x}^n_i}^{m-1}
&\;\le\; \epsilon^{1-m}n^{m(1-\alpha)} + \epsilon\abs{\tilde{x}^n_i}^{m}\,,\\
n^{1 - \alpha}\abs{\tilde{x}^n_i}^{m-1} 
&\;\le\; \epsilon^{1-m} n^{m(1-\alpha)} + \epsilon\abs{\tilde{x}^n_i}^{m}\,, \\
\order(n^{1-\alpha})\order(\abs{\tilde{x}^n_i}^{m-2}) &\;\le\;
\epsilon^{1-\nicefrac{m}{2}}n^{\nicefrac{m(1-\alpha)}{2}}  +  \epsilon\big(\order(\abs{\tilde{x}^n_i}^{m-2})\big)^{\nicefrac{m}{m-2}}\,. 
\end{aligned}
\end{align}
Note that when $\alpha > 1$, $n^{m(1-\alpha)} \le 1$\,. 
Thus, by the expression of $g_n[f_n]$ in \cref{EL3.2A}, we obtain \cref{EL3.2B}.
This completes the proof.
\end{proof}

For each $n$, define the function $\Hat{f}_n\in\Cc(\RR^d\times\cK)$ by
$$
 \Hat{f}_n(x,k) \;\df\; f_n(x) + g_n[f_n](x,k)\,.
$$
The norm-like function $\cV_{m,\xi}$ is defined by $\cV_{m,\xi}(x) \df
\sum_{i\in\cI}\xi_i\abs{x_i}^m$ for $x\in\RR^d$, with $m>0$ and a
positive vector $\xi$ defined in \cref{EL3.1A}.
Let $\widehat{\cL}^{z^n}_n$ be the generator of $(\Hat{X}^n,J^n)$.
Under any stationary Markov policy, by using $\Hat{x}^n$ in \cref{D3.2},
we can write $\widehat{\cL}^{z^n}_n$ as 
$$
\left[\widehat{\cL}^{z^n}_nf(\cdot,\cdot)\right](\Hat{x}^n(x),k) \;=\; \left[\widetilde{\cL}^{z^n}_nf(\Hat{x}^n(\cdot),\cdot)\right](x,k) 
$$
for $f\in\Cc_b(\RR^d\times\cK)$, and define the operator 
$\Delta\widehat{\cL}^{z^n}_{n,k} \colon \Cc_b(\RR^d\times\cK)
\mapsto \Cc_b(\RR^d \times \cK)$ satisfying
$$
\left[\Delta\widehat{\cL}^{z^n}_{n,k}f(\cdot,\cdot)\right](\Hat{x}^n(x),k) 
\;=\; \left[\Delta\cL^{z^n}_{n,k}f(\Hat{x}^n(\cdot),\cdot)\right](x,k)
$$ 
for $f\in\Cc_b(\RR^d\times\cK)$. Let $\widehat{\cV}_{m,\xi}$ be the function defined by
\begin{equation*}
\widehat{\cV}^n_{m,\xi}(x,k) \;\df\;  \cV_{m,\xi}(x) + \frac{1}{n^{\alpha}}\sum_{k'\in\cK}c_{kk'}\Delta\widehat{\cL}^{{z^n}}_{n,k'}\cV_{m,\xi}(x)
\end{equation*}
for $x\in\RR^d$ with the constants $c_{kk'}$ defined in \cref{EL3.2A}.

\begin{theorem}\label{T3.1}
	Let $\widehat{\cL}^{\tilde{z}^n}_n$ denote the generator of the $(\Hat{X}^n,J^n)$ 
	under the scheduling policy defined in \cref{D3.1}. 
	For any even integer $m\ge2$, there exists $n_2\in\NN$ such that
	\begin{equation}\label{ET3.1B}		
			\widehat{\Lg}^{\tilde{z}^n}_n \widehat{\cV}^n_{m,\xi}(\Hat{x},k) \;\le\; \widetilde{C}_1 -  \widetilde{C}_2 \widehat{\cV}^n_{m,\xi}(\Hat{x},k) \,,
			\quad \forall (\Hat{x},k)\in\fX^n\times\cK\,,\quad \forall n> n_2\,,
	\end{equation}
	for some positive constants $\widetilde{C}_1$, $\widetilde{C}_2$ and $n_2 \ge n_0$ depending on $\xi$ and $m$.
	As a consequence, $(\Hat{X}^n,J^n)$ under the the scheduling policy $\tilde{z}^n$ is geometrically ergodic, and for any $m>0$,
	\begin{equation}\label{ET3.1C}	
	\sup_{n\ge n_2}\limsup_{T\rightarrow\infty}\frac{1}{T}\Exp^{\tilde{z}^n}
	\left[\int_0^T\abs{\Hat{X}^n(s)}^m\,\D{s}\right] \;<\; \infty\,.
	\end{equation}
\end{theorem}
\begin{proof}
Since  operators defined in \cref{D3.3,D3.4} are linear, 
we have
$$
\widehat{\cV}^n_{m,\xi}(\Hat{x}^n(x),k) \;=\; n^{-m\beta}\Hat{f}_n(x,k)
$$
for $x\in\ZZ^d_+$ and $k\in\cK$.
Thus, it suffices to show that
\begin{equation}\label{ET3.1A}
	\widetilde{\Lg}^{\tilde{z}^n}_n \Hat{f}_{n}(x,k) \;\le\; \widetilde{C}_1 n^{m\beta} - \widetilde{C}_2 \Hat{f}_n(x,k) \,,
	\quad \forall {(x,k)}\in\ZZ^d_+ \times \cK\,, \quad \forall n\ge n_2\,.
\end{equation}
Let $\xi$ be the vector in \cref{EL3.1A}. 
It is easy to see that
\begin{align}\label{PT3.1A}
\begin{aligned}
	\widetilde{\Lg}^{\tilde{z}^n}_n \Hat{f}_{n}(x,k) 
	&\;=\; \Bar{\cL}^{\tilde{z}_n}_{n,k}f_n(x) + \Lg^{\tilde{z}_n}_{n,k}g_n[f_n](x,k) \\
	&\;\le\; C_1n^{m\beta} - C_2f_n(x) + \Lg^{\tilde{z}_n}_{n,k}g_n[f_n](x,k)\,,
	\quad \forall n  \ge n_0\,,
\end{aligned}
\end{align}
where the inequality follows from \cref{L3.1}.
Applying \cref{L3.2}, we see that there exist
positive constants $C_6$,  $C_7$, and $\tilde{n}_1$,  such that 
\begin{align}\label{PT3.1}
	C_7n^{m\beta} - C_6\Hat{f}_n(x,k) &\;\ge C_1 n^{m\beta} - C_2f_n(x)\,,\quad \forall\; n>\tilde{n}_1\,, \quad \forall (x,k)\in\RR^d\times\cK\,.
\end{align}
Thus, to prove \cref{ET3.1A}, by using \cref{PT3.1,PT3.1A}, 
it suffices to show that, for large enough $n$, 
\begin{equation}\label{PT3.1B}
\Lg^{\tilde{z}_n}_{n,k}g_n[f_n](x,k) \;\le\; C_8n^{m\beta} + \epsilon f_n(x)\,,
\end{equation}
where $C_8$ is some positive constant, and $\epsilon > 0$ can be chosen arbitrarily small for large enough $n$.
Recall the definition of $g_n[f_n]$ in \cref{EL3.2A}, and observe that
\begin{align*}
	\frac{\Delta{\Lg}^{\tilde{z}_n}_{n,k}f_n(x)}{n^{\alpha}} \;=\; \sum_{i \in \cI}\frac{\xi_i}{n^{\alpha}}
	\Big(\Delta{\lambda}_i^n(k) - \Delta{\mu}_i^n(k) \tilde{z}^n_i(x) - 
	\Delta{\gamma}_i^n(k)\tilde{q}^n_i(x) \Big)
	\Big(m(\tilde{x}^n_i)^{m-1} + \order(\abs{\tilde{x}^n_i}^{m-2})\Big)\,.
\end{align*}
Let 
$$
	h_n(x) \;\df\; \frac{1}{n^{\alpha}}\sum_{i\in\cI}\xi_i\tilde{q}^n_i(x)(\tilde{x}^n_i)^{m-1}\,.
$$
Note that in order to prove \cref{PT3.1B}, by using \cref{PL3.2A} and the balance equation $\tilde{z}^n_i(x) =  \tilde{x}^n_i(x) - \tilde{q}^n_i(x) + n\rho_i$, 
we only need to show that
\begin{equation}\label{PT3.1C}
\Lg^{\tilde{z}_n}_{n,k}  h_n(x) \;\le\; C_{9}n^{m\beta} + \epsilon f_n(x)\,,
\end{equation}
where $C_{9}$ is some positive constant, and $\epsilon>0$ can be chosen arbitrarily small for all large enough $n$;  the  other terms in 
$\Lg^{\tilde{z}_n}_{n,k}g_n[f_n]$ can be treated similarly. We obtain
\begin{align*}
\begin{aligned}
\Lg^{\tilde{z}_n}_{n,k}  h_n(x) \;=\; 
\sum_{i\in\cI}\Big(F^{(1)}_{n,i}(x) + F^{(2)}_{n,i}(x)\Big)\,, 
\end{aligned}
\end{align*}
where
\begin{multline*}
	F^{(1)}_{n,i}(x) \;\df\; n^{-\alpha}\xi_i\lambda^n_i(k)	\big(\tilde{q}_i^n(x+e_i)(\tilde{x}^n_i+1)^{m-1} -\tilde{q}^n_i(x)(\tilde{x}^n_i)^{m-1}\big)\\ +n^{-\alpha}\xi_i\big(\mu^n_i(k)\tilde{z}^n_i(x) + \gamma^n_i(k)\tilde{q}^n_i(x)\big)\big(\tilde{q}_i^n(x-e_i)(\tilde{x}^n_i-1)^{m-1}  -\tilde{q}^n_i(x)(\tilde{x}^n_i)^{m-1}\big)\,,
\end{multline*}
and
\begin{multline*}
F^{(2)}_{n,i}(x) \;\df\; n^{-\alpha}\sum_{j\neq i}\big[\lambda^n_i(k)\xi_j	\big(\tilde{q}_j^n(x+e_i) -\tilde{q}^n_j(x)\big)(\tilde{x}^n_j)^{m-1}\\ +\big(\mu^n_i(k)\tilde{z}^n_i(x) + \gamma^n_i(k)\tilde{q}^n_i(x)\big)\xi_j\big(\tilde{q}_j^n(x-e_i) -\tilde{q}^n_j(x)\big)(\tilde{x}^n_j)^{m-1}\big]\,.
\end{multline*}
Note that for $i'\in\cI$,
$$
\abs{\tilde{q}^n_{i'}(x \pm e_i) - \tilde{q}^n_{i'}(x)} \;\le\; 1 \,,
$$
and $\tilde{q}^n_{i'}(x )$ is the unscaled queuing process.
We first consider $F^{(1)}_{n,i}(x)$. We have
\begin{multline*}
\sum_{i\in\cI}F^{(1)}_{n,i}(x) \;\le\;  n^{-\alpha}\sum_{i\in\cI}\Big[\xi_i\lambda^n_i(k)\Big(\tilde{q}_i^n(x+e_i)\big((\tilde{x}^n_i+1)^{m-1}- (\tilde{x}^n_i)^{m-1}\big)  + \abs{\tilde{x}^n_i}^{m-1}\Big) \\
+ \xi_i \Big(\mu^n_i(k)(\tilde{x}^n_i + n\rho_i - \tilde{q}^n_i(x)) + \gamma^n_i(k)\tilde{q}^n_i(x)\Big)
\Big(\tilde{q}_i^n(x-e_i)\big((\tilde{x}^n_i-1)^{m-1}- (\tilde{x}^n_i)^{m-1}\big)  + \abs{\tilde{x}^n_i}^{m-1}\Big)\Big] \,. 
\end{multline*}
Note that 
$$
\left[n\rho_i - \Biggl(n - \sum_{j=1}^{i-1}n\rho_j\Biggr)^+\right]^+ \;=\; 0 \,.
$$
By using the fact that
$
a^+ - b^+ =  \eta(a - b)
$,
for $a,b\in\RR^d$ and $\eta \in [0,1]$, we have
\begin{align}\label{PL3.1D}
\begin{aligned}
\tilde{q}^n_i(x) 
&\;=\; -\left[n\rho_i - \Biggl(n - \sum_{j=1}^{i-1}n\rho_i\Biggr)^+\right]^+ + \left[x_i - \left(n - \sum_{i'=1}^{i-1}x_{i'}\right)^+ \right]^+ \\
&\;=\; -\eta_i(x)(n\rho_i - x_i) + \Bar{\eta}_i(x)\sum_{j=1}^{i-1}(x_j - n\rho_j) \\
&\;=\; \eta_i(x)\tilde{x}^n_i + \bar{\eta}_i(x)\sum_{j=1}^{i-1}\tilde{x}^n_j
\quad \forall x\in\RR^d\,,
\end{aligned} 
\end{align}
for the mappings $\eta_i,\bar{\eta}_i\colon \RR^d\mapsto [0,1]^d$.
By using \cref{PL3.1D} and Young's inequality, we have
\begin{align*}
\tilde{q}_i^n(x\pm e_i)\big((\tilde{x}^n_i\pm 1)^{m-1}- (\tilde{x}^n_i)^{m-1}\big)
&\;\le\; \order\big(\abs{\tilde{x}^n_i}^{m-1}\big) + \sum_{i'=1}^{i-1}\order\big(\abs{\tilde{x}^n_{i'}}^{m-1}\big) \,,\\
\tilde{q}^n_i(x)\tilde{q}_i^n(x - e_i)\big((\tilde{x}^n_i- 1)^{m-1}- (\tilde{x}^n_i)^{m-1}\big)
&\;\le\; \order\big(\abs{\tilde{x}^n_i}^{m}\big) + \sum_{i'=1}^{i-1}\order\big(\abs{\tilde{x}^n_{i'}}^{m}\big)\,.
\end{align*}
Therefore, applying inequalities in \cref{PL3.2A}, we obtain
\begin{equation}\label{PT3.1D}
\sum_{i\in\cI}F^{(1)}_{n,i}(x) \;\le\; C_{10}n^{m(1-\alpha)} + \epsilon f_n(x)\,.
\end{equation}
where $C_{10}$ is some positive constant, and $\epsilon$ can be chosen
arbitrarily small for large enough $n$.
On the other hand, 
since
$
 \tilde{z}^n_{i'}(x)\le x_{i'},$ and $
\tilde{q}^n_{i'}(x)\le x_{i'}
$ for $i'\in\cI$,
 applying Young's inequality, we obtain 
\begin{align}\label{PT3.1E}
\begin{aligned}
F^{(2)}_{n,i}(x) &\;\le\;n^{-\alpha}\sum_{j\neq i}\xi_j\big(\abs{\lambda^n_i(k)} + (\abs{\mu^n_i(k)}+ \abs{\gamma^n_i(k)})(\tilde{x}^n_i + n\rho_i)\big)\abs{\tilde{x}^n_j}^{m-1} \\
&\;\le\; C_{11}n^{m(1-\alpha)} + \epsilon f_n(x)\,,
\end{aligned} 
\end{align}
where $C_{11}$ is some positive constant, and $\epsilon$ can be chosen arbitrarily small for large enough $n$, and the second inequality follows from \cref{PL3.2A}. 
By \cref{L3.2}, there exists  $n_1>0$ such that for all $\Hat{x} \in \fX^n$, $k\in\cK$ and $n>n_1$, we have 
\begin{equation}\label{PT3.1F}
\widehat{\cV}^n_{m,\xi}(\Hat{x},k) \;\ge\; 
\frac{1}{2}\cV_{m,\xi}(\Hat{x}) + \sorder(1).
\end{equation}
We choose $n_2\in\NN$ satisfying $n_2 \ge \max\{n_0,\tilde{n}_1,n_1\}$. Thus, since $1-\alpha < \beta$, by \cref{PT3.1D,PT3.1E} we have shown \cref{PT3.1C}. As a result, we have proved \cref{ET3.1A}, which implies \cref{ET3.1B}.

Let $\Exp^{\tilde{z}_n} = \Exp$.
By It\^o's formula, we obtain
\begin{align*}
\Exp\left[\widehat{\cV}^n_{m,\xi}(\Hat{X}^n(T),J^n(T))\right] -  \Exp\left[\widehat{\cV}^n_{m,k}(\Hat{X}^n(0),J^n(0))\right]
&\;=\; \Exp\left[\int_{0}^{T}\widehat{\cL}^{\tilde{z}^n}_n\widehat{\cV}^n_{m,\xi}(\Hat{X}^n(s),J^n(s))\,\D{s}\right]\,.
\end{align*}
Then, using \cref{ET3.1B}, we have, for $\forall n\ge n_2$,
\begin{equation}\label{PT3.1G}
-\Exp\left[\widehat{\cV}^n_{m,k}(\Hat{X}^n(0),J^n(0))\right] \;\le\; \widetilde{C}_1T - \widetilde{C}_2\Exp\left[\int_{0}^{T}\widehat{\cV}^n_{m,\xi}(\Hat{X}^n(s),J^n(s))\,\D{s}\right]\,.
\end{equation}
Applying \cref{PT3.1F,PT3.1G},
we obtain that, for some positive constant $C_{12}$, and for all $n \ge n_2$,
$$
\frac{1}{T} \Exp\left[\int_{0}^{T} \sum_{i\in\cI}\abs{\Hat{X}^n_i(s)}^m \,\D{s}\right] \le C_{12}\,.
$$
This proves  \cref{ET3.1C}.
\end{proof}
\begin{definition}\label{D3.5}
	Let  $ \upomega \colon 
	\RR^d_+ \mapsto \ZZ^d_+$ 
	be a measurable map defined by
	$$\upomega(x) \;\df\; \left(\lfloor x_1\rfloor,\dots,\lfloor x_{d-1}\rfloor, 
	\lfloor x_{d}  \rfloor + \sum_{i=1}^d(x_i - \lfloor x_i \rfloor)\right)\,.$$
\end{definition}	

\begin{definition}\label{D3.6}
	Let  $v^n : \RR^d  \mapsto \Act$  be any sequence of functions 
	satisfying $v^n(\Hat{x}^n(x))= e_d$, for all $x \notin \fA^n $,
	and such that $x\mapsto v^n(\Hat{x}^n(x))$ is continuous.
	Define 
	\[q^n[v^n](x) \;\df\;
	\begin{cases} 	
	\upomega \big((\langle e,x\rangle - n)^+{v}^n(\Hat{x}^n(x))\big) 
	\quad &\text{for } \sup_{i\in\cI}|x_i-n\rho_i| \le \kappa n \,,\\
	\tilde{q}^{n}(x) \;	&\text{for } \sup_{i\in\cI}|x_i-n\rho_i| > \kappa n \,, 
	\end{cases}
	\]	
	where $\tilde{q}^{n}(x)$ is as in \cref{D3.1}, and $\kappa < \inf_{i\in\cI}\{\rho_i\}$. Define the admissible scheduling policy 
	$$z^n[v^n](x) \;\df\; x- q^n[v^n](x)\,.$$
\end{definition}

We have the following lemma on stabilization of the diffusion-scaled queueing processes. 
\begin{lemma}\label{L3.3}
	For the scheduling policy $z^n$ in \cref{D3.6},
	let $\widehat{\Lg}^{z^n}_n$ denote the infinitesimal generator of $(\Hat{X}^n,J^n)$.  Then, the analogous Foster--Lyapunov equation as in \cref{T3.1} holds.
\end{lemma}

\begin{proof}
	
Observe that for all $n$:
	\begin{enumerate}
		\item[\textup{(i)}] For $i\in\cI$, 
		there exists a constant $C$ such that $\abs{q^n[v^n](x\pm e_i)- q^n[v^n](x)} \le C$;
		\item[\textup{(ii)}]For $i \in \cI$, there exists functions $\epsilon^n_i,\; \tilde{\epsilon}^n_i : \Rd \mapsto [0,1] $ such that $$q_i^n[v^n](x) \;=\; \epsilon^n_i(x) (x_i-n\rho_i) + \tilde{\epsilon}^n_i(x) \sum_{i'=1}^{i-1}(x_{i'}-n\rho_{i'}) + \order(n^{\beta}) \,.$$
	\end{enumerate}
	Hence the same proof as that of \cref{T3.1} may be employed to obtain the  result.
\end{proof}
\begin{remark}
	\cref{L3.3} shows that any sequence of scheduling policies which satisfies (i) and (ii) in the proof of \cref{L3.3} is ``stabilizing''.
\end{remark}

\section{Asymptotic Optimality}\label{S4}

\subsection{Optimal solutions to the limiting diffusion control problems}

The  characterization of optimal control for the limiting diffusion follow from the known results: the discounted problem in  \cite[Section~3.5.2]{ABG12}
and the ergodic problem in  \cite[Sections~3 and 4]{ABP15}.
We summarize these for our model.

\begin{theorem}\label{T4.1}
	$\Hat{V}_{\vartheta}$ 
	is the minimal nonnegative solution in $\Cc^2(\RR^d)$ of 
	\begin{equation*}
	\min_{u\in\Uadm}\left[\cL_u\Hat{V}_{\vartheta}(x) + \rc(x,u) \right] \;=\; \vartheta\Hat{V}_{\vartheta}(x)\,.
	\end{equation*}
	Moreover, $v\in\Usm$ is optimal for the $\vartheta$-discounted problem if and only if
	\begin{equation*}
	\langle b(x,v(x)),\grad\Hat{V}_{\vartheta}(x)\rangle
	+ \rc(x,v(x)) \;=\; H(x,\grad\Hat{V}_{\vartheta}(x))\,,
	\end{equation*}
	where 
	$$
		H(x,p) \;\df\; \min_{u\in\Act} \left[\langle b(x,u),p \rangle + \rc(x,u) \right]\,.
	$$
\end{theorem}
\begin{proof}
The result follows directly from Theorem 3.5.6 and Remark 3.5.8 in  \cite[Section 3.5.2]{ABG12}. 
\end{proof}

\begin{theorem}\label{T4.2}
There exists $\Hat{V}\in\Cc^2(\RR^d)$ satisfying 
$$
 \min_{u\in\Uadm}\left[\cL_u\Hat{V}(x) + \rc(x,u) \right] \;=\; \varrho_*\,.
$$
Also, $v\in\Usm$ is optimal for the ergodic control problem associate with $\rc$
if and only if it satisfies
	\begin{equation*}
	\langle b(x,v(x)),\grad\Hat{V}(x)\rangle
	+ \rc(x,v(x)) \;=\; H(x,\grad\Hat{V}(x))\,.
	\end{equation*}
Moreover, for an optimal $v\in\Usm$, it holds that
$$
\lim_{T\rightarrow\infty}\frac{1}{T}\Exp^v_x\left[\int_0^T\rc(\Hat{X}(s),v(\Hat{X}(s)))\right] 
\;=\; \varrho_*\,, \quad \forall x\in\RR^d \,.
$$
\end{theorem}
\begin{proof}
This follows directly from Theorem 3.4 in \cite{ABP15}.  
\end{proof}

The next theorem shows the existence of an $\epsilon$-optimal control, for
any $\epsilon >0$.
This is proved via the spatial truncation technique. 
\begin{theorem}\label{T4.3}
For any $\epsilon >0$, there exists a ball $B_R$ with $R=R(\epsilon)>0$,
a continuous precise 
control $v_\epsilon\in\Ussm$ which agrees with $e_d $ on $B^c_R$, and an associated invariant measure $\nu_{\epsilon}$ satisfying
$$
\int_{\RR^d}^{}\rc(x,v_{\epsilon}(x))\,\nu_\epsilon(\D{x}) \;\le\; \varrho_* + \epsilon\,.
$$
\end{theorem}
\begin{proof}
This result follows from the proof of claim (5.14) in \cite{ABP15}.
\end{proof}

\subsection{Asymptotic optimality of the discounted cost problem}
In this subsection, we first establish an estimate for $\Hat{X}^n$ 
by using an auxiliary process. Then, following a similar approach as in \cite{AMR04}, we prove asymptotic optimality for the discounted problem.

Given the admissible scheduling policy $\Hat{U}^n$, 
let $\breve{X}^n$ be a $d$--dimensional process defined by
\begin{multline}\label{breveX}
	\breve{X}_i^n(t) \;\df\; \Hat{X}_i^n(0) + \Hat{W}_i^n(t) 
	- \int_{0}^{t}\Bar{\mu}_i^n
	\big(\breve{X}_i^n(s) - \langle e,\breve{X}^n(s) \rangle^{+}\Hat{U}_i^n(s)\big)\,\D{s} \\
	- \int_{0}^{t}\Bar{\gamma}_i^n\langle e,\breve{X}^n(s) \rangle^{+}\Hat{U}_i^n(s)\,\D{s}
\end{multline}
for $i\in\cI$. 

\begin{lemma}\label{LL4.1}
As $n\rightarrow\infty$, $\breve{X}^n$ and $\Hat{X}^n$ are asymptotically equivalent. 
\end{lemma}
The proof of \cref{LL4.1} is given in \cref{S6}.

\begin{lemma}\label{LL4.2}
	We have 
	\begin{equation}\label{ELL4.2A}
	 \Exp\left[\norm{\Hat{X}^n(t)}^{m}\right] \;\le\;  C_1(1 + t^{m_0})(1 + \norm{x}^{m_0})
	\end{equation} 
	for some positive constants $C_1$ and $m_0$, with $m$ defined in \cref{E2.3A}. 
\end{lemma}

\begin{proof}
Recall $\breve{X}^n$ defined in \cref{breveX}. 
For $t\ge 0$, $\breve{X}^n(t) - \langle e,\breve{X}^n(t) \rangle^+\Hat{U}^n(t)$ satisfies the work--conserving condition. Thus, following the same method in \cite[Lemma 3]{AMR04}, we have 
$$ \Exp\left[\norm{\breve{X}^n(t)}^{m}\right] \;\le\;  C_2(1 + t^{m_0})(1 + \norm{x}^{m_0}) $$
for some positive constants $C_2$ and $m_0$. 
As a consequence, \cref{ELL4.2A} holds by \cref{LL4.1}.
\end{proof}

\begin{proof}[Proof of \cref{T2.2}.] (Sketch)
We first show that
\begin{equation}\label{lower_bound_discounted}
\liminf_{n\rightarrow\infty}\;
\Hat{V}^n_{\vartheta}(\Hat{X}^n(0))\;\ge\;\Hat{V}_{\vartheta}(x)\,.
\end{equation}
\cref{LL4.2} corresponds to \cite[Lemma 3]{AMR04}, and \cref{T2.1} corresponds \cite[Lemma 4]{AMR04}.
By using \cref{T4.1}, we can get the same result as in \cite[Proposition 5]{AMR04}.
Thus, we can prove \eqref{lower_bound_discounted} by following the proof of \cite[Theorem 4 (i)]{AMR04}. 

Next, we show that there exists a sequence of admissible scheduling polices $\Hat{U}^n$ which attains optimality (asymptotically).
Observe from \cite{AMR04} that the partial derivates of $\Hat{V}_{\vartheta}$ in \cref{T4.1} up to
order two are locally H\"older continuous 
(see also \cite[Lemma 3.5.4]{ABG12}), and the optimal value $\Hat{V}_{\vartheta}$ has polynomial growth. 
By \cite[Theorem 1]{AMR04}, there exists an optimal control $v_h\in\Usm$ for the discounted problem. Recall $\upomega$ defined in \cref{D3.5}.
Let 
$$
	\fA^n_h \;\df\; \{x \in \RR^d_+\colon \langle e,x \rangle \le x_i,\;\forall i\in\cI \}\,, \quad \text{and} \quad \fX^n_h \;\df\; \{\Hat{x}^n(x)\colon x\in\fA^n_h\}\,.
$$  
Given $X^n$, we construct a sequence of scheduling policies as follows: 
\begin{equation}\label{PT2A}
 Q^n(t) \;\df\; \begin{cases}
 \upomega\big(\langle e,X^n(t) \rangle -n)^{+}v_h(\Hat{X}^n(t))\big) \quad &\text{for}\; \Hat{X}^n(t)\in\fX^n_h\,,\\
 \tilde{z}^n(X^n(t)) \quad &\text{otherwise}\,,
 \end{cases}
\end{equation} 
where $\tilde{z}^n$ is the static priority policy defined in \cref{D3.1}.
Here the value of the scheduling policy outside $\fX^n_h$ is irrelevant for our purpose.
For $n\in\NN$, let $\big(X^n_h,Q^n_h,Z^n_h,\Hat{U}^n_h\big)$ be a sequence of queueing systems constructed by using \cref{PT2A}, and $K^n$ be the process defined by 
$$
K^n(t) \;\df\; \big\langle b(\Hat{X}^n_h(t),\Hat{U}^n_h(t)), \grad \Hat{V}_\vartheta(\Hat{X}^n_h(t)) \big\rangle
 + \rc(\Hat{X}^n_h(t),\Hat{U}^n_h(t)) - H(\Hat{X}^n_h(t),\grad \Hat{V}_\vartheta(\Hat{X}^n_h(t)))\,.
$$
It is easy to see that, for any $y\in\RR^d$, 
$
 \abs{\upomega(y) - y} \le 2d
$. Then, using \cref{T2.1,LL4.2}, and following the same proof as in \cite[Theorem 2 (i)]{AMR04}, we have 
\begin{equation}\label{PL2B}
	\int_0^\cdot \E^{-\vartheta s}K^n(s)\,\D{s} \;\Rightarrow\; 0\,.
\end{equation}
Note that \cref{PL2B} corresponds to the claim (49) in \cite{AMR04}.
Then, we follow the method in \cite[Theorem 4 (ii)]{AMR04} and obtain that
$$
\lim_{n\rightarrow\infty}\;
\mathfrak{J}^n_{\vartheta}\big(\Hat{X}_h^n(0), \Hat{U}_h^n\big)\;\le\;\Hat{V}_{\vartheta}(x) \,.
$$
This completes the proof. 
\end{proof}

\subsection{Proof of the lower bound for the ergodic problem}\label{S4.3}
We have the following theorem concerning the lower bound. 
\begin{theorem}[lower bound]\label{T2.3} 
It holds that
	$$\liminf_{n\rightarrow\infty}\;\Hat{V}^n(\Hat{X}^n(0))\;\ge\;\varrho_*(x)\,. $$ 
\end{theorem}

We first assert that $\Hat{X}^n$ is a semi--martingale.
The proof of the following lemma is given in \cref{S6}.
\begin{lemma}\label{semi-martingale}
	Under any admissible policy $Z^n$, 
	$\Hat{X}^n$ is a semi--martingale with respect to the filtration $\mathbb{F}^n \df \{\mathcal{F}^n_t\colon t\ge0\}$, where $\mathcal{F}^n_t$ is defined in \cref{S2.1}. 
\end{lemma}

\begin{definition}\label{D4.1}
Define the family of  operators $\cA^n_k\colon \Cc^2(\RR^d\times\Act) \mapsto \Cc^2(\RR^d\times \Act\times \cK)$ by
\begin{equation*}
\cA^n_kf(x,u) \;\df\;  \sum_{i\in \cI}\Big(b^n_i(x,u,k)\partial_if(x) + \frac{1}{2}\upsigma^n_i(x,u,k)\partial_{ii}f(x)\Big)\,,
\end{equation*}	
where the functions  $b^n_i,\upsigma_i^n \colon \RR^d \times \Act \times  \cK \mapsto \RR$ are defined by
$$ b^n_i(x,u,k)\;\df\; \ell^{n,k}_i - \mu_i^{n}(k)(x-{\langle e,x\rangle}^{+} u_i ) - \gamma_i^{n}(k){\langle e,x\rangle}^{+}u_i\,, $$
with 
$$\ell^{n,k}_i \df n^{-\beta}[(\lambda^n_i(k) - n\lambda_i^n(k)) - n\rho_i(\mu_i^n(k) -\mu_i(k))]\,,$$
 and 
$$\upsigma^n_{i}(x,u,k) \;\df\; n^{1-2\beta}\mu_i^n(k)\rho_i + \frac{\la_i^n(k)}{n^{2\beta}} +
\frac{\mu^n_i(k)(x_i-{\langle e,x \rangle}^{+} u_i ) +
	\gamma^n_i(k){\langle e,x\rangle}^{+} u_i}{n^{\beta}}  $$
for $i\in\cI$ and $k\in\cK$, respectively. 
\end{definition}
Let  $G^n$ denote the $k$--dimensional process 
$$
G^n_k(t)  \;\df\; \Ind(J^n(t) = k) - \Ind(J^n(0) = k)\,, \quad t\ge 0\,,
$$
and $B^n$ denote the $d$--dimensional processes defined by
\begin{equation}\label{ES4.2A} 
B^n_i(t) \df n^{-\nicefrac{\alpha}{2} + \delta_{0}}\sum_{k\in\cK}(\lambda_i(k) - \rho_i\mu_i(k)) \left[(G^n(t))'\Upsilon\right]_k \,, \quad t\ge0\,,
\end{equation}
for $i\in\cI$, $k\in\cK$, where $\delta_0 \df (1 - \beta) -\frac{\alpha}{2}$.
Then we have the following result, 
which shows that all the long--run average absolute moments of the diffusion--scaled process are finite.
The proof is given in \cref{S6}. 
\begin{lemma}\label{L4.1}
Under any sequence
of admissible scheduling polices $\{\Hat{U}^n\colon n\in \NN\}$ 
such that \newline 
$\sup_n\mathfrak{J}^n(\Hat{X}^n(0),\Hat{U}^n) <\infty$,
we have
\begin{equation}\label{EL4.1A}
\sup_n \limsup_{T\rightarrow\infty} \frac{1}{T}\Exp^{\Hat{U}^n}\left[\int_{0}^{T} \abs{\Hat{X}^n(s)}^m\,\D{s}\right]
\;<\; \infty 
\end{equation}
for $m$ defined in \cref{E2.3A}. 
\end{lemma}
\begin{definition}\label{D4.2}
	Define the mean empirical measure $\zeta^n_T\in\cP(\RR^d\times\Act)$ associated with $\Hat{X}^n$ and $\Hat{U}^n$ by
	$$ \zeta^n_T(A\times B) \;\df\; \frac{1}{T}\Exp\left[\int_{0}^{T}\Ind_{A\times B}(\Hat{X}^n(s),\Hat{U}^n(s))\,\D{s}\right]  $$
	for any Borel sets $A\subset\RR^d$ and $B\subset\Act$. 
\end{definition}

Note that the sequence $\{\zeta^n_T\}$ is tight by \cref{L4.1}. 
The next lemma shows
that the sequence $\{\zeta^n_T\}$ converges,
along some subsequence, to an ergodic occupation measure associated with the limiting diffusion process under some stationary stable Markov control.

\begin{lemma}\label{L4.2}	
Suppose under some sequence of admissible scheduling polices $\{\Hat{U}^n\colon n\in \NN\}$,
\cref{EL4.1A} holds.
Then $\uppi$ is in $\eom$, 
where $\uppi\in\cP(\RR^d\times\Act)$ is any limit point of ${\zeta}^n_T$ as 
$(n,T)\rightarrow\infty$.
\end{lemma}
\begin{proof}
We construct a related stochastic process $\tilde{X}^n$ to prove this lemma.
Let $\tilde{X}^n$ be the $d$--dimensional process defined by
\begin{equation}\label{PL4.2}
\tilde{X}^n \df \Hat{X}^n + B^n\,, 
\end{equation}	
where $B^n$ is defined in \eqref{ES4.2A}.
Applying Lemma 3.1 in \cite{ABMT} and \cref{semi-martingale}, 
$\tilde{X}^n$ is also a semi--martingale.
We first consider the case with $\alpha \le 1$. 
Using the Kunita--Watanable formula for semi--martingales 
(see, e.g., \cite{Protter05}, Theorem II.33) with $\Exp = \Exp^{\Hat{U}^n}$, we obtain 	
	\begin{align}\label{PL4.2A}
	&\;\frac{\Exp\big[f(\tilde{X}^n(T)) - f(\tilde{X}^n(0))\big]}{T} \nonumber \\ 
	&\;=\; 
	\frac{1}{T}\Exp\left[\int_{0}^{T}\sum_{k\in\cK}\cA^n_k\, f(\tilde{X}^n(s),\Hat{U}^n(s))\Ind(J^n(s)=k)\,\D{s}\right] 
	 \nonumber
	\\ &\; \qquad +
	\frac{1}{T}\Exp\left[\sum_{i\in\cI}\int_{0}^{T} \partial_i f(\tilde{X}^n(s))\,\D \Hat{L}^n_i(s)\right] + \frac{1}{T}\Exp\left[\sum_{i\in\cI}\int_{0}^{T} 
	\partial_i f(\tilde{X}^n(s))\,\D B_i^n(s) \right]   \nonumber \\
	& \qquad+ \frac{1}{T}\Exp\left[\sum_{i,i'\in\cI}\int_{0}^{T} \partial_{ii'} f(\tilde{X}^n(s))\,\D\,[B^n_i,B^n_{i'}](s) \right]  + \frac{1}{T}\Exp\bigg[\sum_{s\le T} \cD f(\tilde{X}^n,s)\bigg] 
	\end{align}
for
any $f\in \Cc^{\infty}_c(\RR^d)$, where 
\begin{equation*}
\mathcal{D}f(\tilde{X}^n,s) \;\df\; \Delta f(\tilde{X}^n(s)) - \sum_{i\in\cI}\partial_if(\tilde{X}^n(s-))\Delta\tilde{X}^n_i(s) - \frac{1}{2}\sum_{i,i'\in\cI}\partial_{ii'}f(\tilde{X}^n(s-))\Delta \tilde{X}^n_i(s)\Delta \tilde{X}^n_{i'}(s)
\end{equation*}
for $s\ge0$.
Using  \cite[Lemma 3.1]{ABMT}, $B_i^n(s)+ \Hat{L}_i^n(s)$ is a martingale, and hence the sum of the second and third terms of \cref{PL4.2A} is equal to zero.
By equation (8) in \cite{ABMT} and 
the same calculation as in equation (10) of \cite{ABMT}, the fourth term on the r.h.s. of \cref{PL4.2A} can be written as 
\begin{equation*}
\frac{2}{T}\Exp\left[\sum_{i,i'\in\cI}\int_{0}^{T}\sum_{k\in\cK}\sum_{k'\in\cK} \big(\la_i(k) - \rho_i \mu_i(k)\big) \big(\la_{i'}(k') - \rho_{i'} \mu_{i'}(k')\big)\Upsilon_{kk'}\partial_{ii'}f(\tilde{X}^n(s))\Ind(J^n(s) = k) \,\D{s} \right]\,.
\end{equation*}
Note that for any $f\in\Cc^{\infty}_c(\RR^d)$,
\begin{multline*}
\limsup_{(n,T) \to \infty}\frac{1}{T} 
\Exp\left[\int_{0}^{T} f(\tilde{X}^n(s),\Hat{U}^n(s))\Big(\Ind(J^n(s)=k) - \pi_k\Big)\,\D{s} \right] \\
\;=\; \limsup_{(n,T) \to \infty}\frac{1}{T}
\Exp\left[\int_{0}^{T} n^{-\nicefrac{\alpha}{2}}f(\tilde{X}^n(s),\Hat{U}^n(s))\,
\D \left( \int_{0}^{s}n^{\nicefrac{\alpha}{2}}\big(\Ind(J^n(u)=k) - \pi_k\big)\,\D{u} \right)\right] \;=\; 0
\end{multline*}
by the boundedness of $f$, \cite[Proposition 3.2]{ABMT} and \cite[Theorem 5.2]{JMTW17}.
Thus, we can replace $\Ind(J^n(s)=k)$ by $\pi_k$ for all $k\in\cK$ in \cref{PL4.2A}, when we let $(n,T)\rightarrow\infty$.

We next prove that the last term on the r.h.s. of \cref{PL4.2A} vanishes as
$(n,T)\rightarrow\infty$.
Let
$$
\norm{f}_{\Cc^3} \;\df\; \sup_{x\in\RR^d}\Big(\abs{f(x)} +
 \sum_{i,j\in\cI}\abs{\partial_{ij}f(x)} + \sum_{i,j,k\in\cI}\abs{\partial_{ijk}f(x)}\Big)\,.
$$
Since the jump size of $\tilde{X}^n$ is of order $n^{-\nicefrac{\alpha}{2} + \delta_0}$ or $n^{-\beta}$, 
then by Taylor's formula, we have
\begin{equation}\label{PL4.2C}
\abs{\cD f(\tilde{X}^n,s)} \;\le\;
\frac{\Hat{c}_0\norm{f}_{\Cc^3}}{n^{\nicefrac{\alpha}{2}}}\sum_{i,i'\in\cI}\abs{\Delta\tilde{X}^n_i(s)\Delta\tilde{X}^n_{i'}(s)}
\end{equation}
for some positive constant $\Hat{c}_0$ independent of $n$. 
By equation (2) in \cite{ABMT}, and the independence of Poisson processes, we obtain
\begin{multline}\label{PL4.2D}
\frac{1}{T}\Exp\bigg[\sum_{s\le T}\sum_{i,i'\in\cI}\abs{\Delta\tilde{X}^n_i(s)\Delta\tilde{X}^n_{i'}(s)} \bigg]\\
= \frac{1}{T}\Exp\bigg[\int_0^{T}\sum_{k\in\cK}
\sum_{k\neq k',k'\in\cK}\Hat{c}_k\Big(q_{kk'}\Ind(J^n(s)=k) + q_{k'k}\Ind(J^n(s)=k') \Big)
 \\
+ 
\sum_{i\in\cI}\bigg(\frac{\lambda_i(J^n(s))}{n^{2\beta}} +
\frac{\mu_i^n(J^n(s))Z^n_i(s)}{n^{2\beta}} + \frac{\gamma_i^n(J^n(s))Q^n_i(s)}{n^{2\beta}}\bigg) \,\D{s}\bigg] \,,
\end{multline}
where $\{\Hat{c}_{k}\colon k\in\cK\}$ are determined by the constants in \cref{ES4.2A}.
Using \cref{EL4.1A}, the r.h.s. of \cref{PL4.2D} is uniformly bounded over $n\in\NN$ 
and $T>0$. Therefore, by \cref{ES4.2A,PL4.2C}, 
the last term on the r.h.s. of \cref{PL4.2A} converges to $0$ as
$(n,T)\rightarrow\infty$. 

As in \cref{D4.2},  let $\tilde{\zeta}^n_T \in\cP(\RR^d\times\Act)$ denote the mean empirical measure associated with $\tilde{X}^n$ and $\Hat{U}^n$, that is, 
$$ \tilde{\zeta}^n_T(A\times B) \;\df\; \frac{1}{T}\Exp\left[\int_{0}^{T}\Ind_{A\times B}(\Tilde{X}^n(s),\Hat{U}^n(s))\,\D{s}\right]  $$
	for any Borel sets $A\subset\RR^d$ and $B\subset\Act$. 
Then, by \cref{PL4.2A} and the above analysis, for $f\in\Cc^{\infty}_c(\RR^d)$, we have
\begin{equation}\label{PL4.2E}
 \limsup_{(n,T) \to \infty} \int_{\RR^d\times\Act}\bigg(\sum_{k\in\cK}\cA^n_k\, f(x,u) \pi_k + 
 \Ind(\alpha \le 1)\sum_{i,i'\in\cI}\theta_{ii'}\partial_{ii'}f(x)\bigg)\,\tilde{\zeta}^n_T(\D{x},\D{u})  \;=\; 0\,.
\end{equation}
Note that for $i\in\cI$, $\sum_{k\in\cK}b^n_i(x,u,k)\pi_k$ and $\sum_{k\in\cK}\upsigma_i^n(x,u,k)\pi_k$ converge uniformly over compact sets in $\RR^d\times\Act$, to  $b_i$ (see \eqref{ET2.1A}) and $2\Ind(\alpha\ge1)\lambda^{\pi}_i$, respectively.

On the other hand, by the definition of $B^n$,  we have that
\begin{equation}\label{PL4.1E}
\sup_t\,\babs{\Hat{X}^n(t) - \tilde{X}^n(t)}\;\le\; n^{-\nicefrac{\alpha}{2}} \Bar{C}_0 
\end{equation}
for some positive constant $\Bar{C}_0$.
By \cref{EL4.1A,PL4.1E}, we deduce that $\big\{\tilde{\zeta}^n_T\big\}$ is tight.
Let $(n_l,T_l)$ be any sequence such that $\tilde{\zeta}^n_T$ converges to $\tilde{\uppi}$, as $(n_l,T_l)\rightarrow\infty$. 
Hence, for any $f\in\Cc^{\infty}_c(\RR^d)$, we have
$$\int_{\RR^d\times\Act} \cL_u f(x)\,{\tilde{\uppi}}(\D{x}\times \D{u}) = 0 \quad \text{for } \alpha\le1\,,$$
with $\cL_u$ defined in \cref{def-Lu}.
Using \cref{PL4.1E}, we obtain that
$\zeta^n_T$ and $\tilde{\zeta}^n_T$ have same limit points.
Therefore, as $(n,T)\rightarrow\infty$, any limit point $\uppi$ of $\zeta^n_T$ satisfies
$$\int_{\RR^d\times\Act} \cL_u f(x)\,{{\uppi}}(\D{x}\times \D{u}) = 0 \quad \text{for } \alpha\le1\,.$$
When $\alpha >1$, the proof is the same as above. 
This completes the proof.
\end{proof}

\begin{proof}[Proof of \cref{T2.3}.]
Without loss of generality, suppose $\Hat{V}^{n_l}(\Hat{X}^{n_l}(0))$ for some increasing sequence $\{n_l\}\subset\NN$ converges to a finite value, as $l\rightarrow\infty$, and $\Hat{U}^{n_l}\in\widehat{\Uadm}^{n_l}$.
By the definitions of $\Hat{V}^n$, and the mean empirical measure $\zeta^n_T$ in \cref{D4.2}, there exists a sequence of $\{T_l\}\subset\RR_+$ with $T_l\rightarrow\infty$, such that
$$ 
\Hat{V}^{n_l}(\Hat{X}^{n_l}(0)) + \frac{1}{l}\;\ge\;
 \int_{\RR^d\times\Act}\rc(x,u)\,\zeta^{n_l}_{T_l}(\D{x},\D{u}) \,.
$$   
By \cref{L4.1} and \cref{L4.2},  $\{\zeta^{n_l}_{T_l}\colon l\in\NN\}$ is tight and any limit point of $\zeta^{n_l}_{T_l}$ is in $\eom$.
Thus
\begin{equation*}
\lim_{l\rightarrow\infty} \Hat{V}^{n_l}(\Hat{X}^{n_l}(0))\;\ge\; \int_{\RR^d\times\Act} \rc(x,u)\,{\uppi}(\D{x},\D{u}) \;\ge\; \varrho_*\,.
\end{equation*}
This completes the proof.
\end{proof}

\subsection{Proof of the upper bound for the ergodic problem}\label{S4.4}

We have the following theorem concerning the upper bound. 

\begin{theorem}[upper bound]\label{T2.4}
It holds that
	$$\limsup_{n\rightarrow\infty}\;\Hat{V}^n(\Hat{X}^n(0))\;\le\;\varrho_*(x)\,. $$	\end{theorem}

The following lemma is used in the proof of the upper bound. The lemma shows that
 under a scheduling policy constructed from the $\epsilon$-optimal control  given in \cref{T4.3}, any limit of the mean empirical measures of the diffusion-scaled queuing processes is the ergodic occupation measure of the limiting diffusion under that control.

\begin{lemma}\label{L4.6}
For any fixed $\epsilon > 0$,
	let $\{\Hat{q}^n\colon n\in\NN\}$ 
	be a sequence of maps such that 
	$$
	\Hat{q}_i^n[v](\Hat{x}) \;=\; \begin{cases} 
	\upomega \bigl(\langle e,n^{\beta}\Hat{x}\rangle^+v(\Hat{x})\bigr)  \quad 
	&\text{for } \quad \sup_{i\in\cI}{\abs{\Hat{x}_i}} \le \kappa n^{1-\beta}\,, \\
	\tilde{q}^n(n^{\beta}\Hat{x} + n\rho)
	\quad 
	&\text{for } \quad
	\sup_{i\in\cI}{\abs{\Hat{x}_i}} > \kappa n^{1-\beta}\,.
	\end{cases}  
	$$
	with $\tilde{q}^n$ defined in \cref{D3.1}, $\kappa$ in \cref{D3.6}, and 
	$v \equiv v_\epsilon$ in $\cref{T4.3}$. 
	For $\Hat{x}\in\RR^d$,  let
	$\Hat{z}^n[v](\Hat{x}) = n^{\beta}\Hat{x} + n\rho - \Hat{q}^n[v](\Hat{x}) $,
	and 
	\begin{equation*}
	u^n[v](\Hat{x}) \;\df\; \begin{cases}
	\frac{\Hat{q}^n[v](\Hat{x})}{\langle e,\Hat{q}^n[v](\Hat{x}) \rangle} \quad\quad &\text{if}\; \quad \langle e,\Hat{q}^n[v](\Hat{x}) \rangle > 0\,,\\
	e_d & \text{otherwise}\,. 
	\end{cases}
	\end{equation*}
	Let $\Hat{\zeta}^n_T$ be
	the mean empirical measure defined by
	$$
	\Hat{\zeta}^n_T(A\times B) \;\df\; \frac{1}{T}\Exp^{Z^n}
	\left[\int_0^{T}\Ind_{A\times B}(\Hat{X}^n(s),u^n[v](\Hat{X}^n(s)))\,\D{s}\right] 
	$$
	for Borel sets $A \subset \RR^d$ and $B \subset \Act$,
	where  $\Hat{X}^n$ is the queuing process under  the  admissible scheduling policy
	$Z^n(t) = \Hat{z}^n[v](\Hat{X}^n(t))$.
	Let $\uppi_v\in\cP(\RR^d\times\Act)$ be the ergodic occupation measure of the controlled diffusion in \cref{ET2.1A} under the control $v$. 
	Then $\Hat{\zeta}^n_T$ has a unique limit point $\uppi_v$ as $(n,T)\rightarrow\infty$.
\end{lemma}

\begin{proof}
Applying \cref{L3.3}, we obtain that $\Hat{\zeta}^n_T$ is tight.
Recall the definition of $\tilde{X}^n$ in \cref{PL4.2}. 
Define the mean empirical measure 
$$
\breve{\zeta}^n_T(A\times B) \;\df\; \frac{1}{T}\Exp
\left[\int_0^{T}\Ind_{A\times B}(\tilde{X}^n(s),u^n[v](\tilde{X}^n(s)))\,\D{s}\right] 
$$
for Borel sets $A \subset \RR^d$ and $B \subset \Act$.
For any $f\in\Cc_c^{\infty}(\RR^d\times\Act)$, we have
$$
\frac{1}{T}\Exp
\left[\int_0^{T}f(\tilde{X}^n(s),u^n[v](\tilde{X}^n(s)))\,\D{s}\right] \;=\;
\int_{\RR^d\times\Act}^{}f(x,u)\,\breve{\zeta}^n_T(\D{x},\D{u})\,.
$$
By \cref{PL4.1E},
it is easy to see  $\breve{\zeta}_T^n$ is also tight, 
and $\Hat{\zeta}_T^n$ and $\breve{\zeta}_T^n$ have same limits as $(n,T)\rightarrow\infty$. 
Thus, to prove this lemma, it suffices to show that $\breve{\zeta}^n_T$ has the unique limit point $\uppi_v$ as $(n,T)\rightarrow\infty$. 

Note that
\begin{equation}\label{PL4.6C}
\sup_{\Hat{x}\in \RR^d\cap D}\abs{u^n[v](\Hat{x}) - v(\Hat{x})} \;\rightarrow\; 0 \quad \text{as}\;n\rightarrow\infty
\end{equation}
for any compact set $D\subset\RR^d$.
Let $\uppi^n$ be any limit point of $\breve{\zeta}^n_T$ as $T\rightarrow\infty$.
We have
$$
\uppi^n(\D{x},\D{u}) \;=\; \nu^n(\D{x})\,\delta_{u^n[v](x)}(u)\,,\quad \text{where} \; \nu^n(A)\;=\;\lim_{T \to \infty}\frac{1}{T}\Exp
\left[\int_0^{T}\Ind_{A}(\tilde{X}^n(s))\,\D{s}\right]
$$
 for $A \subset \RR^d$.  By \cref{L4.2},
 $\nu^n$ exists for all $n$ and  $\{\nu^n\colon n\in\NN\}$  is tight.
We choose an increasing sequence $n\in\NN$ such that $\nu^n\rightarrow\nu$ in $\cP(\RR^d)$.
For each $n$, let $\tilde{\cA}^n$ be the operator defined by
$$
\tilde{\cA}^n f(x) \;=\; \sum_{k\in\cK}\cA^n_k\, f(x,u^n[v](x)) \pi_k +
\Ind(\alpha \le 1)\sum_{i,i'\in\cI}\theta_{ii'}\partial_{ii'}f(x)\,.
$$
Recall $\cL_v$ defined in \cref{def-Lu} for $v\in\Usm$.
Therefore, we have
\begin{equation}\label{PL4.6E}
\int_{\RR^d}^{}\tilde{\cA}^n f\,\D{\nu^n} - \int_{\RR^d}^{}\cL_v f\,\D{\nu}
\;=\; \int_{\RR^d}^{}\big(\tilde{\cA}^n f - \cL_v f\big)\,\D{\nu^n} +  \int_{\RR^d}^{}\cL_v f\,\big(\D{\nu^n} - \D{\nu}\big)\,.
\end{equation}
By \cref{PL4.6C} and the convergence of $\tilde{\cA}^n$ in \cref{PL4.2E}, 
we have $\tilde{\cA}^n f \rightarrow \cL_v f$ uniformly as $n\rightarrow\infty$;
thus the first term  on the r.h.s.  of \cref{PL4.6E} converges to $0$.
By the convergence of $\nu^n$, the second term of \cref{PL4.6E} also converges to $0$.
Applying \cref{L4.2}, it holds that, for any $f\in\Cc_c^{\infty}(\RR^d\times\Act)$, 
$$
	\int_{\RR^d}^{}\tilde{\cA}^n f\,\D{\nu^n} \;\rightarrow\; 0 \quad \text{as}\quad n\rightarrow\infty\,.
$$
Therefore,
$$
 \int_{\RR^d}^{}\cL_v f\,\D{\nu} \;=\; 0\,, \quad \forall f\in\Cc_c^{\infty}(\RR^d\times\Act)\,,
$$
which implies that $\nu$ is the invariant measure of $\Hat{X}$ defined in \cref{HatX} under the control $v$.
By \cref{PL4.6C}, we obtain $\delta_{u^n[v](\cdot)}(u) \rightarrow \delta_{v(\cdot)}(u)$ in the topology of Markov controls. Define the ergodic occupation measure $\uppi_v\in\cP(\RR^d\times\Act)$ by $\uppi_{v}(\D{x},\D{u})\df \nu(\D{x})\delta_{v(x)}(u)$. 
Then, for $g\in\Cc^{\infty}_c(\RR^d\times\Act)$, we have
\begin{multline}\label{PL4.6F}
\left|\int_{\RR^d\times\Act}^{} g(x,u)\Big(\uppi_v(\D{x},\D{u}) - \uppi^n(\D{x},\D{u})\Big)\right| \;\le\;
\left|\int_{\Act}^{}\left(\int_{\RR^d}^{}g(x,u)(\nu(\D{x}) - \nu^n(\D{x}))\right)\delta_{u^n[v](x)}(u)\right| \\
+ \left|\int_{\Act}^{}\left(\int_{\RR^d}^{}g(x,u)\nu(\D{x})\right)\big(\delta_{u^n[v](x)}(u) - \delta_{v(x)}(u)\big)\right|
\,.
\end{multline}
By the convergence of $\nu^n$, the first term of \cref{PL4.6F} converges to $0$ as 
$n\rightarrow\infty$. Since $\nu$ has a continuous density, then applying \cite[Lemma 2.4.1]{ABG12}, we deduce that the second term of \cref{PL4.6F} converges to $0$ as 
$n\rightarrow\infty$. Thus, $\uppi^n \rightarrow \uppi_v$ in $\cP(\RR^d\times\Act)$.
This completes the proof.
\end{proof}

\medskip

\begin{proof}[Proof of \cref{T2.4}.] 
	Let $\tilde{m} = 2m$ with $m$ defined in \cref{E2.3A}.
	Let $Z^n$ be a scheduling policy such that
	$Z^n(t) = \Hat{z}^n[v_\epsilon](\Hat{X}^n(t))$ with $v_\epsilon$ (together with a positive constant $R(\epsilon)$) defined in \cref{T4.3} and $\Hat{z}^n$ defined in \cref{L4.6}.
	Note that
		\begin{equation*}
			\int_{\RR^d\times\Act}\rc(x,u)\,\uppi_{v_\epsilon}(\D{x},\D{u})
			\;\le\;\varrho_* + \epsilon \,,
		\end{equation*}
	where $\uppi_{v_\epsilon} \in \cP(\RR^d\times\Act)$ is the ergodic occupation measure defined by $\uppi_{v_\epsilon}(\D{x},\D{u}) \df \nu_{\epsilon}(\D{x})\delta_{v_\epsilon(x)}(u)$. 
	Let $z^n(x) = \Hat{z}^n[v_\epsilon](\Hat{x}^n(x))$ for $x\in\ZZ^d_+$, 
	and $c_0 \equiv R(\epsilon)$ in \cref{D3.2}.
	Then, by \cref{L3.3}, 
	there exits $\Hat{n}_0 \in\NN$ such that
	\begin{equation}\label{PT2.4A}
			\widehat{\Lg}^{z^n}_n\widehat{\cV}_{\tilde{m},\xi}(\Hat{x},k) \;\le\; C_1 - C_2\widehat{\cV}_{\tilde{m},\xi}(\Hat{x},k)
			\quad \forall (\Hat{x},k)\in \fX^n\times\cK\,, \quad \forall n\ge \Hat{n}_0\,,
	\end{equation} 
	for some positive constants $C_1$ and $C_2$. 
	Using \cref{PT2.4A}, we can select a sequence of $\{T_n\colon n\in\NN\}$ such that
	$T_n\rightarrow\infty$ as $n\rightarrow\infty$, and 
\begin{equation*}
\sup_{n\ge \Hat{n}_0}\sup_{T\ge T_n}\int_{\RR^d\times\Act}\widehat{\cV}_{\tilde{m},\xi}(\Hat{x},k)
\,\Hat{\zeta}^n_T(\D{\Hat{x}},\D{u}) \;<\; \infty\,.
\end{equation*} 
It follows that $\widetilde{\rc}(x-\Hat{z}^n[v_\epsilon](x))$ is uniformly integrable.
Moreover, by \cref{L4.6}, $\Hat{\zeta}^n_T$ converges in distribution to $\uppi_{v_\epsilon}$.
This completes the proof.
\end{proof}

\appendix

\section{Proofs of \texorpdfstring{\cref{T2.1}}{} and \texorpdfstring{\cref{L3.1}}{}}\label{S5}

\begin{proof}[Proof of \cref{T2.1}]
To prove (i), we fix $\beta = \nicefrac{1}{2}$, and first show that $\Hat{X}^n$ is stochastically bounded (see Definition 5.4 in \cite{PTW07}).
Recall the definition of $\Hat{X}^n$ in \cref{HatX}.
By \cref{E2.7} and \cref{E2.8}, $\{\Hat{\ell}^n_i + \Hat{L}^n_i\colon n\in\NN\}$ is stochastically bounded in $(\DD,\cJ)$.
The predictable quadratic variation processes of $\Hat{S}^n_i$ and $\Hat{R}^n_i$ are defined by
$$\langle \Hat{S}^n_i \rangle (t) \;\df\;  \int_{0}^{t}\mu^n_i(J^n(s))\Bar{Z}_i^n(s)\,\D{s}\,,
\quad  \langle \Hat{R}^n_i \rangle (t) \;\df\;  \int_{0}^{t}\gamma^n_i(J^n(s))\Bar{Q}_i^n(s)\,\D{s}\,, $$ 
respectively. By \cref{E2.2}, we have the crude inequality
$$ 0 \;\le\; \Bar{X}_i^n(t) \;\le\; \Bar{X}_i^n(0) + n^{-1}A^n_i(t)  \,,$$
and thus, by \cref{con-XQZ},
the analogous inequalities hold for $\Bar{Z}_i^n$ and $\Bar{Q}_i^n$. 
Thus, applying Lemma 5.8 in \cite{PTW07} together with \cref{E2.8}, we deduce
that $\{(\Hat{S}^n_i, \Hat{R}^n_i)\colon n\in\NN\}$ is stochastically bounded in $(\DD,\cJ)^2$, and thus $\{\Hat{W}_i^n\colon n\in\NN\}$ is stochastically bounded.  
For each $u\in\Act$, the  map
$$ x \; \mapsto \;c_1\big(x - \langle e,x \rangle^+ u \big) + c_2\langle e,x \rangle^+ u $$
has the Lipschitz property, where $c_1$ and $c_2$ are some positive constants. 
Then, by \cref{A2.1}, we obtain
$$ \norm{\Hat{X}^n(t)} \;\le\; \norm{\Hat{X}^n(0)}+ \norm{\Hat{W}^n(t)} + C\int_{0}^t\big( 1+ \norm{\Hat{X}^n(s)}\big)\,\D{s} $$
for $t \ge 0$ and some constant $C$. 
Therefore, applying Gronwall's inequality, and using the assumption on $\Hat{X}^n(0)$ and Lemma 5.3 in \cite{PTW07}, it follows
that $\{\Hat{X}^n\colon n\in\NN\}$ is stochastically bounded in $(\DD^d,\cJ)$. 
Then, applying the functional weak law of large numbers (Lemma 5.9 in \cite{PTW07}),
we have  
$$ \frac{\Hat{X}^n}{\sqrt{n}} \;=\; \Bar{X}^n -\rho \;\Rightarrow\; {0} \quad \text{in} \quad  (\DD^d,\cJ) \qasq n \to\infty\,,$$
for $t\ge0$. This implies $ \Bar{X}^n \Rightarrow \rho$ in $(\DD^d, \cJ)$ as $n \to\infty$. By \cref{con-XQZ} and \cref{A2.2}, 
we have $\langle e, \bar{Q}^n \rangle = (\langle e, \bar{X}^n \rangle - 1)^+ \Rightarrow 0$ in $(\DD, \cJ)$ as $n\to\infty$.
Since $\Bar{Q}^n \ge 0$, it follows that
$\bar{Q}^n \Rightarrow 0$ and $\bar{Z}^n \Rightarrow \rho$, both in $(\DD^d, \cJ)$ as $n \to\infty$.

We next prove (ii). For $i\in\cI$ and $t\ge0$, $\Hat{A}_i^n$ can be written as 
$$ \Hat{A}^n_i(t) \;=\; 
n^{-\beta}\left(A_{*,i}\Biggl(n \sum_{k\in\cK}\frac{\lambda_i^n(k)}{n}\int_{0}^{t}\Ind(J^n(s)=k) \,\D{s} \Biggr) 
- n\sum_{k\in\cK}\frac{\lambda_i^n(k)}{n}\int_{0}^{t}\Ind(J^n(s)=k) \,\D{s} \right)\,. $$
By \cite[Theorem 5.1]{JMTW17} and \cref{A2.1}, we have
$$  \sum_{k\in\cK}\frac{\lambda_i^n(k)}{n}\int_{0}^{\cdot}\Ind(J^n(s)=k) \,\D{s} 
\;\xrightarrow[]{\text{u.c.p.}}\;
\lambda^{\pi}_i\mathfrak{e}(\cdot)\,, \quad \text{as}\quad n\rightarrow\infty\,,
$$
for $i\in\cI$,  $\xrightarrow{\text{u.c.p.}}$ denotes uniform convergence
on compact sets in probability, and $\mathfrak{e}(t) \df t$ for all $t\ge0$.
Thus, by the FCLT of Poisson martingales and a random change of time 
(see, for example, \cite[Page 151]{Patrick-99}), we have
$$ \Hat{A}^n \;\Rightarrow\; \Ind(\alpha\ge1)\frac{\Lambda}{\sqrt{2}}W_1  \quad \text{in} \quad (\DD^d, \cJ)\, \qasq n \to\infty\,,$$
where $W_1$ is a $d$-dimensional standard Brownian motion.  
Similarly, applying \cref{T2.1} (i), \cite[Theorem 5.1]{JMTW17} and \cref{A2.1},  
\begin{multline*} 
\sum_{k\in\cK}\mu^n_i(k)\int_{0}^{\cdot}\Bar{Z}^n(s)\Ind(J^n(s)=k)\,\D{s} \;=\;
\sum_{k\in\cK}\mu^n_i(k)\int_{0}^{\cdot}(\Bar{Z}^n(s) - \rho_i) \Ind(J^n(s)=k)\,\D{s} \\ + 
\sum_{k\in\cK}\mu^n_i(k)\rho_i\int_{0}^{\cdot} \Ind(J^n(s)=k)\,\D{s} 
\;\xrightarrow{\text{u.c.p.}}\; \lambda_i^\pi\mathfrak{e}(\cdot) \,, \qasq n \to\infty\,,
\end{multline*}
and
$$  \sum_{k\in\cK}\gamma^n_i(k)\int_{0}^{\cdot}\Bar{Q}^n(s)\Ind(J^n(s)=k)\,\D{s} 
\;\xrightarrow{\text{u.c.p.}}\; 0\,, \qasq n \to\infty\,,$$
for $i\in\cI$ and $t\ge0$. Thus, we  obtain 
$$
\Hat{S}^n \;\Rightarrow\; \Ind(\alpha\ge1)\frac{\Lambda}{\sqrt{2}} W_2\quad 
\text{in} \quad (\DD^d, \cJ)\, \qasq n \to\infty\,,
$$
with a $d$-dimensional standard Brownian motion $W_2$,
and 
$$
\Hat{R}^n \;\Rightarrow\; 0 \quad \text{in} \quad (\DD^d, \cJ)\, \qasq n \to\infty\,.
$$
Since the Poisson processes are independent and the random time changes converge to deterministic functions,
the joint weak convergence of $(\Hat{L}^n, \Hat{A}^n, \Hat{S}^n,\Hat{R}^n)$ holds. 
Note that $\widetilde{W}$, $W_1$ and $W_2$ are independent, and thus
\begin{equation*}
\Hat{W}^n \;\Rightarrow\; \Hat{W} \quad \text{in} \quad (\DD^d, \cJ) \qasq n \to \infty\,.
\end{equation*}
This completes the proof of (ii).

It is easy to see that $\Hat{\ell}^n_i$, $\mu_i^n(k)$ and $\gamma_i^n(k)$ 
are uniformly bounded in $i$, $k$ and $n$.
The rest of the proof of (iii) is same as \cite[Lemma 4(iii)]{AMR04}.

Finally, we prove (iv).
Note that $\Hat{U}^n$ may not have a limit in the space $\DD^d$. 
So to establish the weak limit, we need to assume $\Hat{U}^n$ is tight in $\DD^d$. 
By the representation of $\breve{X}^n$ in \cref{breveX} together with \cref{T2.1} (ii)
, and the continuity of the integral representation 
(see \cite[Theorem 4.1]{PTW07} for one-dimension and \cite[Lemma 4.1]{JMTW17} in the multi-dimensional case), 
any limit of $\breve{X}^n$ is a unique strong solution of \cref{ET2.1A}.
Applying \cref{LL4.1}, we deduce that the limit $\Hat{X}$ of $\Hat{X}^n$ is also a strong solution of \cref{ET2.1A}. 

Recall that $\tau^n(t)$ is defined in \cref{D2.1}. For $r\ge0$, we observe that 
\begin{align*}
\Hat{W}^n_i(t+r)  - \Hat{W}^n_i(t) \;=&\;  \Hat{W}^n_i(\tau^n(t) + r) - \Hat{W}^n_i(\tau^n(t))  \\
&\;+ \Hat{W}^n_i(t+r) -  \Hat{W}^n_i(\tau^n(t) + r) + \Hat{W}^n_i(t) - \Hat{W}^n_i(\tau^n(t))\,.
\end{align*}
It is easy to see that as $n\rightarrow\infty$, $\tau^n(t) \Rightarrow t$. By the random change of time lemma in \cite[Page 151]{Patrick-99}, we have 
$$ \Hat{W}^n_i(t+r) -  \Hat{W}^n_i(\tau^n(t) + r) + \Hat{W}^n_i(t) - \Hat{W}^n_i(\tau^n(t)) \;\Rightarrow\; 0 \quad \text{in} \quad \RR\,,$$
and thus  
$$  \Hat{W}^n_i(\tau^n(t) + r) - \Hat{W}^n_i(\tau^n(t)) \;\Rightarrow\; \Hat{W}_i(t+r) - \Hat{W}_i(t) \quad \text{in} \quad \RR \,.$$
Thus, by \cref{D2.1}, and following the proof of Lemma 6 in \cite{AMR04}, 
we deduce that $\Hat{U}^n$ is non-anticipative. 
\end{proof}

\medskip

\begin{proof}[Proof of \cref{L3.1}]
Note that 
$$(a \pm 1)^m - a^m \;=\; \pm ma^{m-1} + \order(a^{m-2})\,, \quad a\in\RR \,.$$
Recall the definition of $\tilde{x}^n$ in \cref{D3.2}. We obtain
\begin{multline*}
\Bar{\Lg}^{\tilde{z}_n}_nf_n(x) \;=\; \sum_{i \in \cI}\xi_i
\Big(\Bar{\lambda}_i^n(m\tilde{x}^n_i\abs{\tilde{x}^n_i}^{m-2} + \order(\abs{\tilde{x}^n_i}^{m-2})) \\
+ \Bar{\mu}_i^n\tilde{z}^n_i(-m\tilde{x}^n_i\abs{\tilde{x}^n_i}^{m-2} + \order(\abs{\tilde{x}^n_i}^{m-2}))
+ \Bar{\gamma}_i^n\tilde{q}^n_i(-m\tilde{x}^n_i\abs{\tilde{x}^n_i}^{m-2} + \order(\abs{\tilde{x}^n_i}^{m-2}))  
\Big)\,.
\end{multline*}
Let
\begin{equation*}
\bar{F}^{(1)}_n(x) \;\df\;  \sum_{i \in \cI}\xi_i\Big(\Bar{\lambda}_i^n +
\Bar{\mu}_i^n \tilde{z}^n_i + \Bar{\gamma}_i^n \tilde{q}^n_i \Big)\order(\abs{\tilde{x}^n_i}^{m-2})\,,
\end{equation*}
and
\begin{equation*}
\bar{F}^{(2)}_n(x) \;\df\; \sum_{i \in \cI}\xi_i
\Big(\Bar{\lambda}_i^n -\Bar{\mu}_i^n\tilde{z}^n_i
- \Bar{\gamma}_i^n \tilde{q}^n_i\Big)m\tilde{x}^n_i\abs{\tilde{x}^n_i}^{m-2}\,.
\end{equation*}
It is easy to see that
$$
\Bar{\Lg}^{\tilde{z}_n}_nf_n(x) \;=\; \Bar{F}^{(1)}_n(x) + \bar{F}^{(2)}_n(x)\,.
$$
From \cref{D3.1,A2.1}, we have
\begin{align}\label{PL3.1A}
\begin{aligned}
\Bar{F}^{(1)}_n(x) &\;\le\; \sum_{i \in \cI}\xi_i \Big(\Bar{\lambda}^n_i+ \Bar{\mu}^n_i x_i + 
\bar{\gamma}^n_i x_i \Big)\order(\abs{\tilde{x}^n_i}^{m-2}) \\
& \;=\; \sum_{i \in \cI}\xi_i\Big(\Bar{\lambda}^n_i + \Bar{\mu}^n_i(\tilde{x}^n_i +  n\rho_i) + 
\Bar{\gamma}^n_i(\tilde{x}^n_i +  n\rho_i)\Big)\order(\abs{\tilde{x}^n_i}^{m-2}) \\
&\;\le\; \sum_{i \in \cI}\Big(\order(n)\order(\abs{\tilde{x}^n_i}^{m-2}) + \order(\abs{\tilde{x}^n_i}^{m-1})\Big) \,.
\end{aligned}
\end{align}
Next, we consider $\Bar{F}^{(2)}_n(x)$. 
By using the balance equation $\tilde{z}^n_i = \tilde{x}^n_i - \tilde{q}^n_i + \rho^n_in$, 
we obtain
\begin{align*}
\Bar{F}^{(2)}_n(x) \;=\; \sum_{i \in \cI}\xi_i
\Big(-\Bar{\mu}_i^n\tilde{x}^n_i + \Bar{\lambda}^n_i - \Bar{\mu}^n_i\rho_in - 
(\Bar{\gamma}_i^n - \Bar{\mu}_i^n)\tilde{q}^n_i \Big)
m\tilde{x}^n_i\abs{\tilde{x}^n_i}^{m-2}\,.
\end{align*}
By \cref{A2.1}, we have
\begin{equation}\label{PL3.1E}
\Bar{\lambda}^n_i - \Bar{\mu}^n_i\rho_in = \order(n^{\beta})\,.
\end{equation}
 Let  $\Breve{c}_1 \df \sup_{i,k,n}\abs{\gamma^n_i(k) - \mu^n_i(k)}$, and $\Breve{c}_2$ be some positive constant such that
\begin{equation*}
\inf_{i\in\cI,k\in\cI,n\in\NN}\;\bigl\{\mu^n_i(k),\gamma^n_i(k)\bigr\} \;\ge\; \Breve{c}_2 \;>\; 0\,.
\end{equation*}
We choose
\begin{equation*}
\xi_1 \;=\; 1\,, \quad \text{and} \quad \xi_i \;=\; \frac{\epsilon_1^m}{d^m} \min_{i'\le i-1}\xi_{i'}\, \quad \text{for}\;i\ge 2\,,
\end{equation*}
where $\epsilon_1 \df \frac{\Breve{c}_1}{8\Breve{c}_2}$\,.
Then, by using \cref{PL3.1E,PL3.1D}, we obtain
\begin{align}\label{PL3.1B}
\begin{aligned}
\Bar{F}^{(2)}_{n}(x) \;\le\; &\sum_{i\in\cI}^{}\big[
- m\xi_i\big((1-\eta_i(x))\Bar{\mu}_i^n + \eta_i(x)\Bar{\gamma}^n_i\big)\abs{\tilde{x}^n_i}^m \\
&\quad + \xi_i\Big(\order(n^{\beta}) - (\Bar{\gamma}^n_i - \Bar{\mu}^n_i)\Bar{\eta}_i(x)\sum_{j=1}^{i-1}\tilde{x}^n_j\Big)m\tilde{x}^n_i\abs{\tilde{x}^n_i}^{m-2}\big]\\
\;\le\;& \sum_{i\in\cI} \xi_i\order(n^{\beta})(\tilde{x}^n_i)^{m-1}
-\frac{3m\Breve{c}_2}{4}\xi_i\abs{\tilde{x}^n_i}^{m} \,,
\end{aligned}
\end{align}
where the proof for the second inequality of \cref{PL3.1B} is same as the proof for the claim (5.12) in \cite{ABP15}.  
Using Young's inequality and since $\beta \ge \nicefrac{1}{2}$, we have 
\begin{align}\label{PL3.1C}
\begin{aligned}
\order(n)\order(\abs{\tilde{x}^n_i}^{m-2}) &\;\le\; \epsilon\big(\order(\abs{\tilde{x}^n_i}^{m-2})\big)^{\nicefrac{m}{m-2}} 
+ \epsilon^{1-\nicefrac{m}{2}}\big(\order(n)\big)^{m\beta} \\
\order(n^{\beta})\order(\abs{\tilde{x}^n_i}^{m-1}) &\;\le\;
\epsilon\big(\order(\abs{\tilde{x}^n_i}^{m-1})\big)^{\nicefrac{m}{m-1}} 
+ \epsilon^{1-m}\big(\order(n^\beta)\big)^{m}
\end{aligned}
\end{align}
for any $\epsilon > 0$. Therefore, by \cref{PL3.1A,PL3.1B,PL3.1C}, we have 
\begin{equation*}
\Bar{\cL}^{\tilde{z}^n}_nf_n(x) \;\le\; C_1n^{m\beta} - C_2f_n(x)\,,\quad \forall x\in\ZZ^d_+\,.
\end{equation*}
This completes the proof.	
\end{proof}

\section{Proofs of \texorpdfstring{\cref{LL4.1}}{}, \texorpdfstring{\cref{semi-martingale}}{} and \texorpdfstring{\cref{L4.1}}{}}\label{S6}

\begin{proof}[Proof of \cref{LL4.1}]
	For $i\in\cI$ and $t\ge 0$, we have
	\begin{multline}\label{PL2.1A}
	\Hat{X}^n_i(t) - \breve{X}^{n}_i(t) \;=\; 
	- \int_0^t(\mu^n_i(J^n(s)) - \Bar{\mu}^{n}_i)\Hat{X}^n_i(s)\,\D{s} + \int_0^t \Bar{\mu}^{n}_i(\breve{X}^n_i(s) - \Hat{X}^n_i(s))\,\D{s}
	\\
	+ \int_0^t(\mu^n_i(J^n(s)) - \Bar{\mu}^{n}_i - \gamma^n_i(J^n(s))  + \Bar{\gamma}^{n}_i)\langle e,\Hat{X}^n(s) \rangle^+
	\Hat{U}_i^n(s)\,\D{s} \\
	- (\Bar{\mu}^{n}_i - \Bar{\gamma}^{n}_i) \int_0^t (\langle e,\breve{X}^n(s) \rangle^+  - \langle e,\Hat{X}^n(s) \rangle^+)\Hat{U}_i^n(s)\,\D{s}\,.
	\end{multline}
	For any $a,b\in\RR$, $a^+ - b^+ = \eta(a - b)$ with $\eta\in [0,1]$. 
	Then, the last term of \cref{PL2.1A} can be written as 
	\begin{equation*}
	\int_0^t \bigl(\langle e,\breve{X}^n(s) \rangle^+  - \langle e,\Hat{X}^n(s) \rangle^+\bigr)\,\Hat{U}_i^n(s)\,\D{s} \;=\; 
	\int_0^t 
	\tilde{\eta}(\breve{X}^n(s),\Hat{X}^n(s)) \langle e,\breve{X}^n(s) -\Hat{X}^n(s)\rangle\Hat{U}_i^n(s)\,\D{s} \,,
	\end{equation*}
	where $\tilde{\eta}(x,y) \colon (x,y)\in\RR^2 \mapsto [0,1]$\,. 
	Note that $ \Hat{U}^n_i(t) \in [0,1]$ for all $i\in\cI$ and $t\ge0$. 
	By the continuous integral mapping (\cite[Lemma 5.2]{PW09}), if the first and third terms of \cref{PL2.1A} converge to the zero process uniformly on compact sets
	in probability, 
	then $\Hat{X}^n - \breve{X}^n$ must converge to the zero process uniformly on compact sets in probability. 
	The first term of \cref{PL2.1A} can be written as 
	$$
	-\sum_{k\in\cK}\mu_i^n(k)\int_{0}^{t}n^{-\nicefrac{\alpha}{2}}\Hat{X}^n_i(s)\,\D \bigg( 
	n^{\nicefrac{\alpha}{2}}\int_0^s(\Ind(J^n(u) = k) - \pi_k)\D{u}\bigg)\,.
	$$
	Note that $1 - \beta - \alpha = \min\{0, (1-\alpha)/2\}$.
	Applying \cref{T2.1} (i), we have $n^{-\nicefrac{\alpha}{2}}\Hat{X}^n_i$ converges to the zero process uniformly on compact sets in probability.
	Similarly, since $\Hat{U}^n_i(s)$ is bounded, by \cref{T2.1} (i), 
	we  obtain that
	$n^{-\nicefrac{\alpha}{2}} \langle e,\Hat{X}^n(s) \rangle^+)\Hat{U}_i^n(s)$
	converges to the zero process uniformly on compact sets in probability.
	Then,  
	the asymptotic equivalence of $\breve{X}^n$ and $\Hat{X}^n$
	follows as in the proof in \cite[Lemma 4.4]{JMTW17}.
\end{proof}

\medskip

\begin{proof}[Proof of \cref{semi-martingale}]
	For $n\in\NN$ and $i\in\cI$, define the processes $M_{S_i} = \{M_{S^n_i}(t)\colon t\ge0 \}$ and $M_{R_i} = \{M_{R_i}(t)\colon t\ge0 \}$ by 
	$$M_{S_i}(t)\df S_{*,i}(t) - t\,, \quad \text{and} \quad 
	M_{R_i} \df R_{*,i}(t) - t\,,$$ respectively. It is obvious that $M_{S_i}$ and $M_{R_i}$
	are square integrable martingales with respect to the filtration generated by the processes $S_{*,i}$ and $R_{*,i}$. 
	Define the $d$--dimensional processes $\uptau^n_1$ and $\uptau^n_2$ by 
	\begin{align*}
	\uptau^n_{1,i}(t) \;\df\; \int_{0}^t\mu^n_i(J^n(s))Z^n_i(s) \D{s} \,, \quad \text{and} \quad
	\uptau^n_{2,i}(t) \;\df\; \int_{0}^t\gamma^n_i(J^n(s))Q^n_i(s) \D{s}\,,
	\end{align*}
	respectively.
	It is easy to see that $\{\uptau^n_{j,i}\colon i\in\cI, j\in\{1,2\}\}$ have continuous nondecreasing nonnegative sample paths.
	For $x_1\in\RR^d_+$ and $x_2\in\RR^d_+$, we  obtain 
	$$
	(\uptau^n_1(t)\le x_1, \uptau^n_2(t) \le x_2) \;\in\; \mathcal{H}^n(x_1,x_2)\,, 
	$$
	where 
	\begin{multline*}
	\mathcal{H}^n(x_1,x_2) \;\df\; \{S^n_i(s_{1,i}),R^n_i(s_{2,i}),X_i^n(0) \colon i\in\cI,s_1\le x_1, s_2\le x_2 \}\\\vee\sigma\{A^n_i(s),J^n(s),Z^n_i(s) \colon s\ge 0, i\in\cI\}\vee\mathcal{N}.
	\end{multline*}
	This implies that $(\uptau^n_1(t),\uptau^n_2(t))$ is $\mathbb{H}^n$--stopping time, where 
	$\mathbb{H}^n \df \{\mathcal{H}^n(x_1,x_2)\colon x_1\in\RR^d_+,x_2\in\RR^d_+\}$.
	Since $X^n_i(t) \le X^n_i(0) + A^n_i(t)$ for $i\in\cI$, we observe that
	\begin{align*}
	\Exp\left[ \uptau^n_{1,i}(t)\right] &\;\le\; \max_k\{\mu_i(k) \}t (X_i^n(0) + \Exp[A^n_i(t)] + n) \;<\; \infty\,, \\
	\Exp\left[S_{*,i}(\uptau^n_{1,i}(t))\right] &\;\le\; \max_k\{\mu_i(k) \}t (X_i^n(0) + \Exp[A^n_i(t)] + n) \;<\; \infty\,.
	\end{align*}
	Similarly, we have that
	$\Exp[\uptau^n_{2,i}(t)]$ and $\Exp[R_{*,i}(\uptau^n_{2,i}(t))]$ are finite.
	Thus, applying Lemma 3.2 in \cite{PTW07} and Theorem 8.7 on page 87 of \cite{Kurtz86}, and using the decomposition in \cref{HatX} and Lemma 3.1 in \cite{JMTW17},
	we  conclude that $\Hat{X}^n$ is a semi--martingale with respect to the filtration $\widetilde{\mathbb{F}}^n \df \{\widetilde{\mathcal{F}}^n_t\colon t\ge0\}$, where
	\begin{equation*}
	\widetilde{\mathcal{F}}^n_t \;\df\; \sigma\{S^n_i(s),R^n_i(s), X_i^n(0)  \colon i\in\cI,s\le t  \}
	\vee\sigma\{A^n_i(s),J^n(s),Z^n_i(s) \colon i\in\cI, s\ge 0\}\vee\mathcal{N}\,,
	\end{equation*}
	and $\mathcal{N}$ is a collection of $\Prob$--null sets. 
	Since the processes $A^n(t)$, $J^n(t)$ and $Z^n(t)$ are adapted to 
	$\mathcal{F}^n_t$, we can replace $\widetilde{\mathbb{F}}^n$ by the smaller filtration $\mathbb{F}^n$.
	This completes the proof.
\end{proof}

\medskip

\begin{proof}[Proof of \cref{L4.1}]
	Define the function $g\in\Cc^2(\RR^d)$ by $g(x) \df \sum_{i\in\cI}g_i(x_i)$ 
	with $g_i \in \Cc^2(\RR)$ defined by $g_i(x)= \abs{x}^m$ for $\abs{x}\ge 1$ and $i\in\cI$.
    Recall $\cA^n_k$ defined in \cref{D4.1}.
	Applying the Kunita--Watanable formula to $\tilde{X}^n$ with $\Exp = \Exp^{\Hat{U}^n}$ and the fact
	$\Hat{L}_i^n + B_i^n$ is a martingale, we have 
	\begin{multline}\label{PL4.4A}
	\Exp\big[g(\tilde{X}^n(t))\big] \;= \;  \Exp\big[g(\tilde{X}^n(0))\big] 
	+ \sum_{k\in\cK}\Exp\left[\int_{0}^{t}\cA^n_k\, g(\tilde{X}^n(s),\Hat{U}^n(s))\Ind(J^n(s)=k)\,\D{s}\right]
	\\+ \Exp\left[\sum_{i,i'\in\cI}\int_{0}^{t} \partial_{ii'} g(\tilde{X}^n(s))\,\D\,[B^n_i,B^n_{i'}](s) \right]
	+ \Exp\Biggl[\sum_{s\le t} \cD g(\tilde{X}^n,s)\Biggr]
	\end{multline}
	for $t\ge 0$, where $\cD g(\tilde{X}^n,s)$ is defined analogously to \cref{PL4.2A}.
	By \cref{A2.1} and Young's inequality, we have 
	\begin{align*}
	b^n_i(x,u,k)g^{\prime}_i(x) &\;\le\; \frac{\bar{c}_{1}}{d}(1+(\langle e,x \rangle^+)^m)- \frac{\Bar{c}_{2}}{d} \abs{x}^m \,,\\
	\sigma^n_i(x,u,k)g^{\prime\prime}_i(x) &\;\le\;\frac{2\bar{c}_{1}}{d}(1+(\langle e,x \rangle^+)^m) + \frac{\Bar{c}_{2}}{4d} \abs{x}^m\,,
	\end{align*}
	and thus, for all $k\in\cK$, we obtain
	\begin{equation}\label{PL4.4B}
	\cA^n_k\,g(x,u)\;\le\; 2\Bar{c}_1(1+(\langle e,x \rangle^+)^m) - \frac{7}{8}\Bar{c}_2 \abs{x}^m\,,
	\end{equation}
	where $\Bar{c}_1$ and $\Bar{c}_2$ are some positive constants independent of $n$.
	Following the same analysis for the fourth term on the r.h.s. of \cref{PL4.2A}, 
	and using Young's inequality,
	we have 
	\begin{equation}\label{PL4.4C}
	\Exp\left[\sum_{i,i'\in\cI}\int_{0}^{t} \partial_{ii'} g(\tilde{X}^n(s))\,\D\,[B^n_i,B^n_{i'}](s) \right]
	\,\le\, \Exp\left[\int_{0}^{t}\bar{c}_1\bigl(1 + (\langle e,\tilde{X}^n(s) \rangle^+)^m\bigr) 
	+ \frac{\bar{c}_2}{8}\abs{\tilde{X}^n(s)}^m\,\D{s} \right]\,.
	\end{equation}
	Since the jump size is of order $n^{-\beta}$ or $n^{-\nicefrac{\alpha}{2} + \delta_0}$, 
	we can find a positive constant $\Bar{c}_3$ such that
	$$\sup_{\abs{x_i-x^{\prime}_i}\le1}\;\abs{g^{\prime\prime}_i(x^{\prime}_i)}
	\,\le\, \Bar{c}_3(1+\abs{x_i}^{m-2})$$
	for each $x_i\in\RR$. Then, applying the Taylor remainder theorem, we obtain
	$$
	\Delta g_i(\tilde{X}_i^n(s)) - g^{\prime}_i(\tilde{X}_i^n(s))\Delta\tilde{X}_i^n(s) \;\le\;
	\frac{1}{2}\sup_{\abs{x^{\prime}_i -\tilde{X}^n_i(s-)}\le1}\abs{g^{\prime\prime}_i(x^{\prime}_i)}(\Delta\tilde{X}_i^n(s))^2\,,
	$$
	for each $i\in\cI$.
	Following a similar analysis as in \cref{PL4.2D}, and using Young's inequality, we obtain  
	\begin{align}\label{PL4.4D}
	\Exp\Biggl[\sum_{s\le t} \cD g_i(\tilde{X}^n,s)\Biggr] 
	\;&\le\; \Exp\Biggl[\sum_{s\le t}\Bar{c}_3\Big(1 + \abs{\tilde{X}^n_i(s-)}^{m-2}\Big)(\Delta \Hat{X}^n_i(s))^2 \Biggr] \nonumber\\
	\;&\le\; \Exp\bigg[ \int_{0}^{t} \Big(\Bar{c}_4 + \Bar{c}_5(\langle e,\tilde{X}^n(s) \rangle^+)^m + \frac{\Bar{c}_2}{2}\abs{\tilde{X}^n(s)}^m\Big)\,\D{s} \bigg]
	\end{align}
	for some positive constants $\Bar{c}_4$ and $\Bar{c}_5$. 
	Thus, by \cref{PL4.4A,PL4.4B,PL4.4C,PL4.4D}, we obtain
	\begin{equation}\label{PL4.4E}
	\Exp\left[\int_{0}^{t}\abs{\tilde{X}^n(s)}^m\,\D{s}\right] \;\le\; \Bar{c}_6\Exp\big[g(\tilde{X}^n(0))\big] 
	+ \Bar{c}_7 t 
	+ \Bar{c}_8 \Exp\left[\int_{0}^{t}(\langle e,\tilde{X}^n(s) \rangle^+)^m \,\D{s}\right]
	\end{equation}
	for some positive constants $\Bar{c}_i$, $i\in\{6,7,8\}$. 
	Using \cref{PL4.1E}, we see that \cref{PL4.4E} also holds if we replace $\tilde{X}^n$
	with $\Hat{X}^n$. 
	Therefore, under any sequence satisfying $\sup_n \mathfrak{J}^n(\Hat{X}^n(0),\Hat{U})<\infty$,  
	we have established \cref{EL4.1A}. This completes the proof.
\end{proof}

\section*{Acknowledgement}
This work is  supported in part by an Army Research Office grant W911NF-17-1-0019,
in part by NSF grants DMS-1715210, CMMI-1635410, and DMS/CMMI-1715875,
and in part by Office of Naval Research through grant N00014-16-1-2956.


\begin{thebibliography}{10}

\bibitem{ABMT}
D.~Anderson, J.~Blom, M.~Mandjes, H.~Thorsdottir, and K.~de~Turck.
\newblock A functional central limit theorem for a {M}arkov-modulated
  infinite-server queue.
\newblock {\em Methodol. Comput. Appl. Probab.}, 18(1):153--168, 2016.

\bibitem{ABP15}
A.~Arapostathis, A.~Biswas, and G.~Pang.
\newblock Ergodic control of multi-class {$M/M/N+M$} queues in the
  {H}alfin-{W}hitt regime.
\newblock {\em Ann. Appl. Probab.}, 25(6):3511--3570, 2015.

\bibitem{ABG12}
A.~Arapostathis, V.~S. Borkar, and M.~K. Ghosh.
\newblock {\em Ergodic control of diffusion processes}, volume 143 of {\em
  Encyclopedia of Mathematics and its Applications}.
\newblock Cambridge University Press, Cambridge, 2012.

\bibitem{AP18}
A.~Arapostathis and G.~Pang.
\newblock Infinite horizon asymptotic average optimality for large-scale
  parallel server networks.
\newblock {\em Stochastic Process. Appl.}, (in press), 2018.

\bibitem{AP16}
A.~Arapostathis and G.~Pang.
\newblock Infinite horizon average optimality of the {N}-network queueing model
  in the {H}alfin-{W}hitt regime.
\newblock {\em Math. Oper. Res.}, 43(3):838--866, 2018.

\bibitem{APS}
A.~Arapostathis, G.~Pang, and N.~Sandri\'c.
\newblock Ergodicity of {L}\'evy--driven {SDE}s arising from multiclass
  many--server queues.
\newblock {\em ArXiv e-prints}, 1707.09674v2, 2018.

\bibitem{AMR04}
R.~Atar, A.~Mandelbaum, and M.~I. Reiman.
\newblock Scheduling a multi class queue with many exponential servers:
  asymptotic optimality in heavy traffic.
\newblock {\em Ann. Appl. Probab.}, 14(3):1084--1134, 2004.

\bibitem{BH12}
Y.~Bakhtin and T.~Hurth.
\newblock Invariant densities for dynamical systems with random switching.
\newblock {\em Nonlinearity}, 25:2937--2952, 2012.

\bibitem{BBMZ15}
M.~Bena{\"i}m, S.~Le~Borgne, F.~Malrieu, and P.-A. Zitt.
\newblock Qualitative properties of certain piecewise deterministic {M}arkov
  processes.
\newblock {\em Ann. Inst. H. Poincar{\'e} Probab. Statist.}, 51(3):1040--1075,
  2015.

\bibitem{Patrick-99}
P.~Billingsley.
\newblock {\em Convergence of probability measures}.
\newblock Wiley Series in Probability and Statistics: Probability and
  Statistics. John Wiley \& Sons, Inc., New York, second edition, 1999.
\newblock A Wiley-Interscience Publication.

\bibitem{BGL14}
A.~Budhiraja, A.~Ghosh, and X.~Liu.
\newblock Scheduling control for {M}arkov-modulated single-server multiclass
  queueing systems in heavy traffic.
\newblock {\em Queueing Syst.}, 78(1):57--97, 2014.

\bibitem{Hairer15}
B.~Cloez and M.~Hairer.
\newblock Exponential ergodicity for {M}arkov processes with random switching.
\newblock {\em Bernoulli}, 21(1):505--536, 2015.

\bibitem{Dieker-Gao}
A.~B. Dieker and X.~Gao.
\newblock Positive recurrence of piecewise {O}rnstein-{U}hlenbeck processes and
  common quadratic {L}yapunov functions.
\newblock {\em Ann. Appl. Probab.}, 23(4):1291--1317, 2013.

\bibitem{Kurtz86}
S.~N. Ethier and T.~G. Kurtz.
\newblock {\em Markov Processes: Characterization and {C}onvergence}.
\newblock Wiley, 1986.

\bibitem{GS12}
D.~Gamarnik and A.~L. Stolyar.
\newblock Multiclass multiserver queueing system in the {H}alfin-{W}hitt heavy
  traffic regime: asymptotics of the stationary distribution.
\newblock {\em Queueing Syst.}, 71(1-2):25--51, 2012.

\bibitem{HW81}
S.~Halfin and W.~Whitt.
\newblock Heavy-traffic limits for queues with many exponential servers.
\newblock {\em Oper. Res.}, 29(3):567--588, 1981.

\bibitem{JMTW17}
H.~M. Jansen, M.~Mandjes, K.~De~Turck, and S.~Wittevrongel.
\newblock Diffusion limits for networks of markov-modulated infinite-server
  queues.
\newblock {\em ArXiv e-prints}, 1712.04251, 2017.

\bibitem{Kha12}
R.~Z. Khasminskii.
\newblock Stability of regime-switching stochastic differential equations.
\newblock {\em Probl. Inf. Transm.}, 48(3):259--270, 2012.
\newblock Translation of Problemy Peredachi Informatsii {{\bf{4}}8} (2012), no.
  3, 70--82.

\bibitem{Kha07}
R.~Z. Khasminskii, C.~Zhu, and G.~Yin.
\newblock Stability of regime-switching diffusions.
\newblock {\em Stochastic Process. Appl.}, 117(8):1037--1051, 2007.

\bibitem{RMH13}
R.~Kumar, M.~E. Lewis, and H.~Topaloglu.
\newblock Dynamic service rate control for a single-server queue with
  {M}arkov-modulated arrivals.
\newblock {\em Naval Res. Logist.}, 60(8):661--677, 2013.

\bibitem{Mao99}
X.~Mao.
\newblock Stability of stochastic differential equations with {M}arkovian
  switching.
\newblock {\em Stochastic Process. Appl.}, 79(1):45--67, 1999.

\bibitem{PTW07}
G.~Pang, R.~Talreja, and W.~Whitt.
\newblock Martingale proofs of many-server heavy-traffic limits for {M}arkovian
  queues.
\newblock {\em Probab. Surv.}, 4:193--267, 2007.

\bibitem{PW09}
G.~Pang and W.~Whitt.
\newblock Heavy-traffic limits for many-server queues with service
  interruptions.
\newblock {\em Queueing Syst.}, 61(2-3):167--202, 2009.

\bibitem{PZ17}
G.~Pang and Y.~Zheng.
\newblock On the functional and local limit theorems for {M}arkov modulated
  compound {P}oisson processes.
\newblock {\em Statist. Probab. Lett.}, 129:131--140, 2017.

\bibitem{Protter05}
P.~E. Protter.
\newblock {\em Stochastic integration and differential equations}, volume~21 of
  {\em Stochastic Modelling and Applied Probability}.
\newblock Springer-Verlag, Berlin, 2005.
\newblock Second edition. Version 2.1. Corrected third printing.

\bibitem{LQA17}
L.~Xia, Q.~He, and A.~S. Alfa.
\newblock Optimal control of state-dependent service rates in a {MAP}/{M}/1
  queue.
\newblock {\em IEEE Trans. Automat. Control}, 62(10):4965--4979, 2017.

\end{thebibliography}
\end{document}